\documentclass{aims} %
\usepackage{amsmath}
\usepackage{mathtools}
\usepackage{graphicx}
\usepackage[colorlinks=true, allcolors=blue]{hyperref}
\usepackage{amsfonts}
\usepackage{appendix}
\usepackage{bbm}
\usepackage{caption}
\usepackage{subcaption}
\usepackage[usenames, dvipsnames, table]{xcolor}
\usepackage{paralist}
\usepackage[misc]{ifsym}
\usepackage{epsfig} %
\usepackage{epstopdf} %
\usepackage[colorlinks=true]{hyperref}
\hypersetup{urlcolor=blue, citecolor=red}
\allowdisplaybreaks

\def\currentvolume{X}
 \def\currentissue{X}
  \def\currentyear{200X}
   \def\currentmonth{XX}
    \def\ppages{X-XX}
     \def\DOI{10.3934/xx.xxxxxxx}

\newtheorem{theorem}{Theorem}[section]

\theoremstyle{definition}

\newtheorem{remark}[theorem]{Remark}

\title[Scalable optimal control for conservation laws]
      {Scalable optimal control for inequality-constrained discretizations of scalar conservation laws} %

      \author[Falko Ruppenthal, Denis Ridzal, Dmitri Kuzmin, Pavel Bochev]{}

\subjclass{Primary: 49M41; Secondary: 65K10,65M60.}
\keywords{Conservation laws, finite element methods, discrete maximum principles, mass conservation, optimal control.}

\thanks{$^*$Corresponding author: Falko Ruppenthal}

\newcommand\red[1]{#1}
\newcommand\blue[1]{#1}
\newcommand\magenta[1]{#1}
\newcommand\PBB[1]{#1}

\newcommand{\R}{\mathbb{R}}

\newcommand{\pd}[2]{\frac{\partial #1}{\partial #2}}

\newcommand{\td}[2]{\frac{\mathrm d #1}{\mathrm d #2}}

\newcommand{\beq}{\begin{equation}}
\newcommand{\eeq}{\end{equation}}

\newcommand {\dx} {\,{\rm d}{\mathbf x}}

\newcommand {\ds} {\,{\rm d}{\mathrm s}}

\DeclareMathOperator*{\minimize}{minimize}
\DeclareMathOperator*{\argmin}{arg\,min}

\begin{document}
\maketitle

\centerline{\scshape
  Falko Ruppenthal$^{{\href{mailto:falko.ruppenthal@math.tu-dortmund.de}{\textrm{\Letter}}}*1}$, Denis Ridzal$^{{\href{mailto:dridzal@sandia.gov}{\textrm{\Letter}}}2}$,}

\centerline{\scshape
  Dmitri Kuzmin$^{{\href{mailto:kuzmin@math.uni-dortmund.de}{\textrm{\Letter}}}1}$, Pavel Bochev$^{{\href{mailto:pbboche@sandia.gov}{\textrm{\Letter}}}2}$}

\medskip

{\footnotesize
 \centerline{$^1$Institute of Applied Mathematics (LS III), TU Dortmund University, Germany}
} %

\medskip

{\footnotesize
  \centerline{$^2$Center for Computing Research, Sandia National Laboratories, Albuquerque, NM, USA}
}

\bigskip

\begin{abstract}
Optimization-based (OB) alternatives to traditional flux limiters couch preservation of properties such as local bounds and maximum principles into optimization problems, which impose these properties through inequality constraints. In this paper, we propose a new `potential-target' OB approach that enforces these properties using an optimal control formulation, in which the control is the source term expressed through flux potentials. The resulting OB formulation combines superb accuracy with excellent local conservation properties, but complicates the development of scalable iterative solvers, which is greatly influenced by the choice of semi-norms for the objective function. We use this fact to design scalable iterative solvers based on matrix-free trust-region Newton methods with projections onto convex sets. These solvers leverage inexpensive multigrid V-cycles while satisfying all constraints to machine precision. Numerical experiments reveal that the convergence behavior of the solvers can be greatly improved by a simple scaling of the inequality constraints. We demonstrate excellent performance in applications to linear test problems, such as $L^2$ projection and solid body rotation, and to the Cahn–Hilliard equation.
  
\end{abstract}

\section{Introduction}
Modern numerical schemes for systems of conservation laws are frequently equipped with mechanisms for enforcing discrete maximum principles and/or positivity of scalar quantities of interest. A variety of flux and slope limiting techniques were developed for this purpose during the last 50 years; see the book \cite{kuzmin-hajduk2022} for \PBB{a comprehensive} review of existing approaches. Limiters are commonly used to constrain sums of `massless' quantities, e.g., numerical fluxes, deviations from cell averages, and element vectors with zero sums, which represent the differences between consistent discretizations of high and low order. Any provable property of the latter can be preserved in the process of limiting at the expense of introducing some numerical diffusion.
As shown by Bochev~et~al.~\cite{bochev2012constrained} in the context of flux limiting and by Aizinger~\cite{aizinger2011} in the context of slope limiting, the aforementioned approaches can be interpreted as exact solvers for optimization problems with simple box constraints. 
\PBB{One can show \cite{peterson2024optimization} that the latter can be viewed as simplifications of the original optimization constraints, based on worst-case considerations.

  Although sufficient for the preservation of the desired bounds, the simplified box constraints ``shrink'' the feasible set of the optimization problem, potentially excluding the optimal solution; see, e.g., the ``torture'' test in \cite{Bochev_11_JCP}. For this reason, recently there has been an} increased interest in the finite element community in
optimization-based (OB) alternatives to traditional limiters \cite{evans2009,liu2023,may2013,van2019}. 
OB methods considered in the present paper trace their origins to the work of Bochev~et~al.~\cite{bochev2012constrained,bochev2020optimization}. Exploiting a close relationship to algebraic flux correction schemes, they provide a unified framework for incorporating inequality constraints into finite element discretizations and conservative data projections. 
\PBB{In this paper, we exploit this framework to develop a new \emph{potential-target} OB formulation that draws upon the ideas of  Ruppenthal and Kuzmin~\cite{ruppenthal2023optimal}.} 
\PBB{The approach in \cite{ruppenthal2023optimal}} is closely related to the elliptic optimal control problems considered by Meyer and Rösch \cite{meyer2004}. To maintain the local mass conservation property while enforcing
the imposed box constraints, a source term control is expressed in terms of numerical fluxes. Each flux is proportional
to a difference of two scalar \emph{flux potentials} (FP), which are treated as optimization variables. The objective function
is defined using potentials that correspond to a standard Galerkin finite element discretization. The corresponding
\emph{target fluxes} remain unchanged if the pointwise box constraints are satisfied for a given pair of nodes. In this way, the action of optimal control is localized to small subdomains in which the unconstrained baseline
scheme would produce undershoots/overshoots. 
\PBB{This is an important improvement over the previous ``density-target'' OB formulations (OBFEM-DD) \cite{bochev2013fast,peterson2024optimization} that may experience some ``mass-spreading'', while the number of discrete unknowns in our new formulation} is less than 
in the locally conservative flux-target OB formulation \cite{bochev2020optimization} \PBB{that does not have this issue.} 
Potential disadvantages of FP-based optimization procedures include significant computational overhead and lack of scalability. This issue is addressed in the present paper, in which we find the cause of efficiency bottlenecks and propose a scalable revised version of the OB-FP method introduced in~\cite{ruppenthal2023optimal}.

While OB procedures are generally more costly than closed-form limiters, \PBB{they yield globally optimal states that are the best possible solutions with respect to the optimization objective and are guaranteed to possess the desired physical properties. In addition,} they can be used to enforce relevant constraints even if the exact solution violates them because of modeling errors. For example, a physical upper bound can be enforced on volume fractions~\cite{ruppenthal2023optimal}. Similarly, approximate solutions to the Cahn-Hilliard equation can be constrained to stay in an admissible range~\cite{liu2023}. We remark that the \emph{physics-aware} limiting algorithms proposed in~\cite{frank2020bound,kuzmin2012} and~\cite[Sec. 4.8]{kuzmin-hajduk2022} can also \PBB{be} interpreted as iterative solvers for optimization problems. Our new \PBB{\emph{potential-target}} OB algorithm \PBB{provides} more efficient customized solvers, \PBB{which further strengthen the case for optimization-based approaches.}

The remainder of the paper is structured as follows. We present the baseline discretization and formulate a general optimal control problem in Sections~\ref{sec:baseline} and~\ref{sec:obfem}, respectively. The proposed revision of the potential-target OB method \PBB{into the new potential-target OB formulation}
is \PBB{described} in Section~\ref{sec:scalable}. In Section~\ref{sec:solver}, we provide relevant implementation details and discuss ways to \PBB{solve the reformulated optimization problems in a computationally efficient manner. Section~\ref{sec:examples} provides numerical examples illustrating the accuracy and scalability of the new algorithms, while  Section~\ref{sec:conclusions} offers some concluding remarks.}

\section{Baseline finite element discretizations}
\label{sec:baseline}

\PBB{In what follows,} $\mathbf{x}=(x_1,\ldots,x_d)^\top$ \PBB{is a vector in $\R^d$,  $d\in\{1,2,3\}$}, and  $t\ge 0$ a time instant. We are interested in simulating the evolution of a scalar conserved
quantity $u(\mathbf{x},t)$ in a domain $\Omega\subset\R^d$ with boundary $\partial\Omega$.
We discretize the \PBB{equations} governing the dynamics of this quantity 
using the continuous Galerkin method \PBB{on a conforming partition $\mathcal T_h$ of $\Omega$ into finite elements $K_i$, $i=1,\ldots,M$, i.e., $\mathcal T_h=\{K_1,\ldots,K_M\}$.}
\PBB{The elements $K_i$ can be $d+1$ simplices, quadrilaterals or hexahedrons. For simplicity, we restrict our attention to meshes comprising elements of the same type and to
linear ($\mathbb{P}_1$) or multilinear ($\mathbb{Q}_1$) finite
element approximations on $K\in\mathcal T_h$.} The corresponding Lagrange basis functions
$\varphi_1,\ldots,\varphi_N$ are associated with the vertices
$\mathbf{x}_1,\ldots,\mathbf{x}_N$
\PBB{of  $\mathcal T_h$. We recall that the basis functions possess the following partition of unity property
\begin{equation}\label{eq:PU}
\sum_{i=1}^{N}\varphi_i\equiv 1 \,.
\end{equation}
}
The coefficients \PBB{$u_j$, $j=1,\ldots,N$} of a finite element approximation\footnote{We overload the notation and denote by $u$ both the exact solution~$u(\mathbf{x},t)$ of the governing equation and the vector~$u\in \mathbb{R}^N$ of the discrete nodal values $u_1,\ldots,u_{N}$. The intended meaning should be clear from the context.}
$$
u_h=\sum_{j=1}^Nu_j\varphi_j
$$
are determined using a (stabilized) weak form of the governing equation.
\smallskip

In this work, we 
consider fully discrete problems that can be written as
\beq\label{galerkin}
M_Cu^{\rm new}=M_Cu^{\rm old}+\Delta t\,r(u^{\rm new},u^{\rm old}).
\eeq
For example, $u^{\rm old}\in\R^N$ and $u^{\rm new}\in\R^N$ may contain
the coefficients $u_j(t) $ of $u_h(\mathbf{x}_j,t)$ at  discrete
time levels $t^{\rm old}\ge 0$ and $t^{\rm new}=t^{\rm old}+\Delta t$, respectively. 

The coefficients  of the consistent mass matrix 
$M_C=(m_{ij})_{i,j=1}^{N}$ are given by
$$
m_{ij}=\PBB{\int_{\Omega}\varphi_i\varphi_j\dx} =  \sum_{e=1}^{M}\int_{K_e}\varphi_i\varphi_j\dx.
$$
If the components of $r(u^{\rm new},u^{\rm old})\in\R^{N}$ sum to zero, then
the partition of unity property \eqref{eq:PU} implies that
$$\sum_{e=1}^{M}\int_{K_e}u_h^{\rm new}\dx=
\sum_{e=1}^{M}\int_{K_e}u_h^{\rm old}\dx.
$$
This global conservation property is preserved if $M_C$ is approximated by the lumped mass matrix $M_L=(\delta_{ij}m_i)_{i,j=1}^{N}$ with positive diagonal entries $$m_i=\sum_{j=1}^{N}m_{ij}.$$
It is also preserved if we replace the \emph{baseline} discretization~\eqref{galerkin} with
 \beq\label{obfem}
 \red{M_L\tilde u^{\rm new}}=M_Lu^{\rm old}+\Delta t\,[r(u^{\rm new},u^{\rm old})
 + \left(M_L - M_C\right) \dot u],
 \eeq
 where \red{$u^{\rm new}$ is the solution of the linear system \eqref{galerkin}} and
 $\dot u$ is a given vector of \emph{flux potentials}
   (named so in \cite{ruppenthal2023optimal}) \red{that defines
  an approximation $\tilde u^{\rm new}$ to $u^{\rm new}$}. Note that systems
 \eqref{galerkin} and \eqref{obfem} are equivalent
 \red{and $\tilde u^{\rm new}=u^{\rm new}$}  for
 $\dot u=(u^{\rm new}-u^{\rm old})/\Delta t$.

\section{Formulation of optimal control problems}
\label{sec:obfem}

A numerical scheme of the form \eqref{obfem} is \emph{bound preserving}
with respect to~an admissible range $[u_i^{\min},u_i^{\max}]\subseteq [u^{\min},u^{\max}]$
if the updated nodal values satisfy 
\beq\label{lmp}
u^{\min}\le u_i^{\min}\le u_i^{\rm new} \le u_i^{\max}\le u^{\max}
\qquad \forall i\in\{1,\ldots,N\}.
\eeq
The choice of the global bounds $u^{\min}$ and $u^{\max}$ is typically
dictated by the physical nature of the problem at hand. For example,
volume and mass fractions must stay between $u^{\min}=0$ and
$u^{\max}=1$. The local bounds $u_i^{\min}$ and $u_i^{\max}$ can be
constructed using the given values of  $u_h^{\rm old}$ at node $i$ and its
nearest neighbors. The purpose of imposing \PBB{such} local discrete maximum principles is to
avoid spurious oscillations within the range $[u^{\min},u^{\max}]$ of
physically admissible values.

\PBB{Optimization-based (OB) methods couch the preservation of these bounds into a constrained optimization problem that minimizes a suitable objective function  $J = J(w)$, $J : \R^n \rightarrow \R$ subject to  constraints enforcing the local bounds. The argument $w$ of this function depends on the type of the OB formulation. In \emph{flux-based} OB methods \cite{Bochev_11_JCP,bochev2020optimization}, this variable represents the difference between a high-order target flux and the optimization state, while in \emph{density-based} methods \cite{peterson2024optimization,bochev2013fast} $w$ measures the mismatch of the state and a high-order nodal density approximation.}

\PBB{For the new potential-target OB formulation in this paper,} $w \in \R^n$ will be a vector of control variables, \PBB{and all optimization problems will have the following general structure:} 
\begin{subequations}
\label{obfem-first}
\begin{align}
\minimize_{w\in\R^{N}} &\;\; J(w) \label{obfem-obj} \\
\text{subject to} &\;\;
M_Lu=M_L\tilde u+\Delta t\,g(w),\quad
\mathbbm{1}^\top g(w)=0, \label{obfem-eq} \\
&\;\;
u_i^{\min}\le u_i\le u_i^{\max} \quad \forall i\in\{1,\ldots,N\}. \label{obfem-ineq}
\end{align}
\end{subequations}
Here $\mathbbm{1} \in \mathbb{R}^n$, $\mathbbm{1}=(1,\ldots,1)^\top$, and $[u_i^{\min},u_i^{\max}]\subseteq [u^{\min},u^{\max}]$.
The objective function in~\eqref{obfem-obj} is \PBB{defined as the regularized mismatch between  $w$ 
and} an optimization target $\tilde w$, i.e.,
\beq
J(w)=\frac12\|w-\tilde w\|_A^2+\frac{\mu}{2}\|w\|_B^2,\label{obfem-last}
\eeq
where $\mu\ge 0$ is a regularization parameter. The
(semi-)norms $\|\cdot\|_A$ and $\|\cdot\|_B$ are defined in terms of positive
(semi-)definite matrices $A,B\in\R^{N\times N}$, respectively, as follows:
$$\|w\|_A^2=w^\top Aw,\qquad \|w\|_B^2=w^\top Bw \qquad
\forall w\in\R^{N}.$$

\PBB{The OB-FP approach of  Ruppenthal and Kuzmin \cite{ruppenthal2023optimal}, which motivates this work,  uses $A=B=M_C$, $\mu=0.01$, equality constraints \eqref{obfem-eq} expressed through a vector of flux potentials $w=\dot u$, and}
$$
\tilde u=u^{\rm old}+\Delta t M_L^{-1}\,r(u^{\rm new},u^{\rm old}),
\qquad g(w)=(M_L - M_C) w.
$$
The choice
$w=(u^{\rm new}-u^{\rm old})/\Delta t$ corresponds to the
baseline scheme~\eqref{galerkin}. The \PBB{resulting OB-FP}
problem yields accurate and bound-preserving approximations, \PBB{yet,} the convergence behavior of iterative solvers for the system of discrete optimality conditions leaves
a lot to be desired. We clarify the reasons for this and
propose a revised formulation of the potential-based
optimal control method in the next section. Existing
relationships to other optimization-based approaches
and to algebraic flux correction schemes are discussed in~\cite{ruppenthal2023optimal}.

\begin{theorem}[Solvability \cite{ruppenthal2023optimal}]
  Let $\rho=\sum_{i=1}^{N}r_i(u^{\rm new},u^{\rm old})$
  and
  $$
    \omega_i=\begin{cases}
  u_i^{\max}-u_i^{\rm old} & \quad \mbox{if}\  \rho>0,\\
  0 & \quad \mbox{if}\  \rho=0,\\
  u_i^{\min}-u_i^{\rm old} & \quad \mbox{if}\  \rho<0.
  \end{cases}
    $$ 
  Assume that
  \beq\label{solvcond}
  \rho\Delta t\le \min_{1\le i\le N}m_i\left(\sum_{j=1}^{N}\omega_j\right).
  \eeq
Then the feasible set of the optimization problem~\eqref{obfem-first}
is nonempty.
\end{theorem}

\begin{proof}
  It suffices to show that the inequalities~\eqref{obfem-ineq} hold.
  Following the proof presented in \cite{ruppenthal2023optimal}, we define
  $c^B=\rho/\left(\sum_{j=1}^{N}\omega_j\right)$ and
  decompose
  $\rho$ into nodal contributions
 $$r_i^B= \frac{\omega_i\rho}{\sum_{j=1}^{N}\omega_j}$$
  such that $\sum_{i=1}^{N}(r_i^B-r_i)=\rho-\rho=0$ and $r_i^B=c^B(u_i^*-u_i^{\rm old})$, where
  $$
  u_i^{*}=\begin{cases}
  u_i^{\min} & \mbox{if}\ \rho<0,\\
  0  & \mbox{if}\ \rho=0,\\
  u_i^{\max} & \mbox{if}\ \rho>0.
  \end{cases}
  $$
  Let $\dot u^B$ be a solution of $g(\dot u^B)=r^B-r$.  Then
  $$u_i^B:=u_i^{\rm old}+\frac{\Delta t}{m_i}(r_i+g_i(\dot u^B))=u_i^{\rm old}+
  \frac{\Delta tr_i^B}{m_i}=u_i^{\rm old}+\frac{c^B\Delta t}{m_i}(u_i^*-u_i^{\rm old})$$ is a convex
  combination of $u_i^{\rm old}\in [u_i^{\min},u_i^{\max}]$ and
  $u_i^*\in[u_i^{\min},u_i^{\max}]$ because $c^B\ge 0$ by definition and 
$c^B\Delta t\le m_i$ by \eqref{solvcond}.
\end{proof}

\begin{remark}
In OB schemes for conservation laws, the auxiliary quantity $\rho$ is a sum of boundary
fluxes. If periodic boundary conditions are imposed, then $\rho=0$, and the solvability
condition \eqref{solvcond} is satisfied. In general, the validity of~\eqref{solvcond} can be enforced by choosing a sufficiently small time step
$\Delta t$.
\end{remark}
\begin{remark}
If there exists a bound-preserving low-order approximation $u^L$ such
that ${\sum_{i=1}^N m_i(u_i^L-u_i^{\rm old})=\Delta t\rho}$, then the
backup $u^B=u^L$ can  be recovered
using $g(\dot u^B)=M_L\frac{u^B-u^{\rm old}}{\Delta t}-r(u^{\rm new},u^{\rm old})$.
For hyperbolic problems with bound-preserving exact solutions, $u^L$
can often be constructed, e.g., using algebraic flux correction tools~\cite{bochev2020optimization,ruppenthal2023optimal}. \red{The optimization
procedure can then be configured to use the backup $\tilde u=u^B$
and a vector $g(w)$ of antidiffusive fluxes in
\eqref{obfem-first}. Such OB formulations
have the important advantage that they produce a consistent
approximation $u=\tilde u$ even in the case $g(w)=0$ and
a more accurate result for optimal choices of $g(w)$.}
\end{remark}

\section{Scalable potential-target \PBB{optimization-based control formulation}}
\label{sec:scalable}

The poor performance of iterative solvers
in~\cite{ruppenthal2023optimal} can be attributed to two root
causes.
One is an inadequate choice of the matrix $A$ for the
objective function.
Another is the direct use of flux potentials as optimization variables;
see~\cite[Remark~8]{ruppenthal2023optimal}.
Introducing
$$g=(M_L-M_C)\dot u$$
and substituting $u=\tilde u+\Delta t\,M_L^{-1}g$
 into the inequality constraints \eqref{obfem-ineq},
we \PBB{reformulate} the optimization problem as follows:  
\begin{subequations}
\label{eq:ocp}
\begin{align}
  \minimize_{g \in \mathbb{R}^N} &\;\; J(g) \\
  \text{subject to} &\;\; 
  \mathbbm{1}^\top g=0, \label{eq:ocp-sum} \\
  &\;\;
  \frac{M_L}{\Delta t} \left( u^{\min} - \tilde u \right) \leq g \leq \frac{M_L}{\Delta t} \left( u^{\max} - \tilde u \right). \label{eq:ocp-bounds}
\end{align}
\end{subequations}
Here and below, inequalities involving vectors are meant to hold pointwise.

The \PBB{``density-state density-target'' (OBFEM-DD)} method proposed by Bochev et. al.~\PBB{\cite{bochev2013fast}}, \cite{bochev2020optimization}
minimizes\footnote{\blue{The superscript `$T$' of a vector
    stands for `target' here and below. The symbol `$\top$' is used
    for transposed quantities.
    The transpose of a vector $u^T$ is denoted by
$(u^T)^\top$.}}
$$
J(g)=\|u-u^T\|_{M_C}^2=\|\tilde u-u^T+\Delta t\,M_L^{-1}g\|_{M_C}^2,
$$
where \PBB{the optimization target} $u^T$ is \PBB{an unconstrained high-order approximation of nodal ``densities''}. This approach leads to efficient algorithms but there is no guarantee
that a solution $\dot u$ of the linear system $(M_L-M_C)\dot u=g$ will
be a good approximation to the vector $\dot u^T$ of potentials that
define the numerical fluxes of a target scheme.

Similarly to the optimization-based method proposed by Liu et al.~\cite{liu2023},
\PBB{OBFEM-DD can only guarantee global mass conservation, which is enforced by a single linear equality constraint.} To achieve better local mass conservation properties, we define
the objective function
\begin{equation}\label{eq:OB-PT-obj}
J(g)=\frac12\|Sg-\dot u^T\|_A^2
\end{equation}
using a vector $\dot u^T$ of target potentials and a linear
operator $S:\R^N\to\R^N$ that solves the discrete Neumann
problem $(M_L-M_C)\dot u=g$ for $\dot u
=Sg$.  By default, we use the semi-norm
$\|\cdot\|_A=\|\cdot\|_{M_L-M_C}$ in the
definition of $J(g)$. The original potential-target
 formulation \cite{ruppenthal2023optimal} uses the $L^2$ norm
$\|\cdot\|_A=\|\cdot\|_{M_C}$. 

 The matrix $M_L-M_C$ has the properties of a discrete Laplacian operator
 and \PBB{has a one dimensional null-space spanned by the constant vector. Thus,}
 $\dot u =Sg$ is unique up to a constant \cite{bochev2005finite}, \PBB{and to}
 ensure the invertibility property
 $$
S[(M_L-M_C)\dot u]=\dot u\qquad\forall \dot u\in\R^N,
$$
we need to impose a linear constraint such as 
$\sum_{i=1}^N\dot u_i=0$
or $\dot u_1=0$. \PBB{In our implementation, we use the latter approach because it amounts to the imposition of a homogeneous Dirichlet condition at a single mesh node. We denote the modified version of the Neumann solution operator $S$ by $\tilde S$.}

We recall that the semi-norm property of $\|\cdot\|_{A}:\R^N\to\R_0^+$ rests on the
triangle inequality and absolute homogeneity properties.
For a symmetric positive semi-definite 
matrix $A\in\R^{N\times N}$, we introduce the scalar product
$$\langle u,v \rangle_A =u^\top A v\qquad\forall u,v\in\R^N ,$$
and define the semi-norm $\|u\|_A \coloneqq \sqrt{\langle u,u \rangle_A}$.
Using the \PBB{Cauchy inequality}, we obtain the estimate
    \begin{align*} 
    \PBB{\| u+v \|^2_A}
    &=\PBB{\| u\|^2_A + 2 \langle v,u \rangle_A + \| v \|^2_A} \\
    &\leq \PBB{\| u\|^2_A + 2\| u\|_A \| v\|_A+ \| v\|^2_A
      = \left(\| u\|_A + \| v\|_A\right)^2.}
      \end{align*}
Taking the square root gives the triangle inequality 
$$\| u + v \|_A \leq \| u \|_A + \| v \|_A\quad ~\forall u,v \in \mathbb{R}^N.$$
Absolute homogeneity of  $\|\cdot\|_{A}$ follows from the fact that 
$$\|s u \|_A = \sqrt{(s u)^\top A (s u)} = \sqrt{s^2 u^\top A u} = |s| \|u\|_A\quad
\forall s\in\R \text{ and } \forall u\in\R^n.$$

To efficiently solve optimal control problems of type~\eqref{eq:ocp},
we use optimization methods that \PBB{exploit} derivative \PBB{information}.
Specifically, the methods that we choose in Section~\ref{sec:solver}
require the gradient vector $\nabla J(g)$ and the application of the
Hessian~$H$ of the objective function \PBB{\eqref{eq:OB-PT-obj}}.
We consider optimal control formulations in which $\|\cdot\|_A$ is chosen to be the
$L^2$ norm, i.e., $A=M_C$, as in~\cite{ruppenthal2023optimal}, and the
discrete Laplacian semi-norm, $A=M_L-M_C$.
For these two formulations, the function evaluations and the requisite
derivative computations are as follows:
\begin{itemize}
\item $L^2$ norm ($A=M_C$),
\begin{equation}
\begin{aligned}
    \label{eq:func-L2}
    J(g) &= \frac{1}{2} \left(S g - \dot u^T\right)^\top M_C \left(S g - \dot u^T\right), \\
    \nabla J(g) &= S^\top M_C \left(S g - \dot u^T\right), \\
    H &= S^\top M_C S; \quad \text{and}
\end{aligned}
\end{equation}
\item discrete Laplacian semi-norm ($A=M_L-M_C$),
\begin{equation}
  \begin{aligned}
    \label{eq:func-semi}
    J(g) &= \frac{1}{2} \left(S g - \dot u^T\right)^\top \left(g - \left(M_L-M_C\right)\dot u^T\right), \\
    \nabla J(g) &= S g - \dot u^T, \\
    H &= S.
  \end{aligned}
\end{equation}
\end{itemize}

We also explore the use of scaling, or `preconditioning,' in the formulation of inequality constraints.
If the bounds in the inequality constraints~\eqref{eq:ocp-bounds} are local, see~\eqref{lmp},
the distance between $\tilde u_i$ and an extremum of its local neighborhood,
which we can write as $\min(\tilde u_i-u_i^{\min},u_i^{\max}-\tilde u_i)$, may
become very small and attain values around machine precision.
To enhance numerical stability, we multiply each inequality
constraint by a component of a positive vector, $d\in\R_+^N$.
Introducing the diagonal matrix $D \coloneqq \left(d_i\delta_{ij}\right)_{i,j=1}^N$,
we define the scaled vector $\tilde g = Dg$ of control variables
and reformulate our constrained optimization problem as follows:
\begin{subequations}
\label{eq:ocp-scaled}
\begin{align}
  \minimize_{\tilde g \in \mathbb{R}^N} &\;\; \tilde J(\tilde g) \coloneqq \frac{1}{2} \left\| \tilde S \tilde g - \dot u^T \right\|^2_{A} \\
  \text{subject to} &\;\;
  \mathbbm{1}^{\top}D^{-1}\tilde g = 0, \\
  &\;\;
  a \coloneqq D \frac{M_L}{\Delta t} \left( u^{\min} - \tilde u \right) \leq
  \tilde g \leq
  D \frac{M_L}{\Delta t} \left( u^{\max} - \tilde u \right) \eqqcolon b.
\end{align}
\end{subequations}
The application of the modified solution operator $\tilde S$ to $\tilde g$ yields $\dot u$ such that
\begin{align*}
    \left( M_L - M_C \right) \dot u = D^{-1}\tilde g=g,\qquad \dot u_1 = 0.
\end{align*}
\PBB{We conclude this section with a few comments related to \eqref{eq:ocp-scaled}.}
First, this `preconditioned' formulation coincides with the
original, \eqref{eq:ocp}, for $d = \mathbbm{1}$.
Second, in practice we often use the scaling $D = M_L^{-1}$.
Third, the function and derivative computations~\eqref{eq:func-L2} and~\eqref{eq:func-semi}
must be modified slightly to accommodate the solution operator~$\tilde S$.
The derivation is left to the reader.

\section{Solvers for the optimal control problem}
\label{sec:solver}

We implement and solve various instances of the optimal control problem~\eqref{eq:ocp-scaled}
using the scientific software libraries MFEM (Modular Finite Element Methods
Library~\cite{Anderson_MFEM_A_Modular_2021}), {\tt hypre}~\cite{hypre},
and ROL (Rapid Optimization Library \cite{ROL2022ICCOPT,rol-website,ridzal2017rapid}), \PBB{which is distributed with} the Trilinos
Project~\cite{heroux2005overview,heroux2012new,trilinos-website}.
We use MFEM for the finite element discretization and the parallel assembly of vectors and
matrices that define problem~\eqref{eq:ocp-scaled},
such as $M_L$, $M_C$, $D$, $\dot u^T$, $a$, and $b$.
All linear systems \PBB{requiring the modified Neumann solution operator $\tilde S$ use} MFEM's conjugate gradient (CG) method and {\tt hypre}'s V-cycle multigrid preconditioner with a Gauss-Seidel smoother \PBB{to compute its action.} Subsequently, we implement ROL's operator-based interface and use its matrix-free optimization
algorithms to solve the optimal control problem.
We now discuss our choices of ROL's solvers.

Optimization problem~\eqref{eq:ocp-scaled} belongs to the class of problems
\[
\minimize_{\tilde g \in \R^n} \;\; \mathcal{J}(\tilde g)
\qquad \text{subject to}  \qquad \tilde g \in \mathcal{C},
\]
where $\mathcal{J} : \R^n \rightarrow \R$ is a smooth function and $\mathcal{C} \subseteq \R^n$
is a nonempty, closed and convex set.
In order to satisfy the constraints as accurately as possible,
we use optimization methods based on projections onto convex sets.
Through the use of projections, we can maintain feasibility, $\tilde g \in \mathcal{C}$,
to within machine precision, while also reducing~$\mathcal{J}(\tilde g)$.
This is a key consideration.
Additionally, since our objective functions depend on applications of the solution
operator $\tilde S$---ideally requiring only \emph{approximate} linear system solves
using multigrid methods---we aim for optimization algorithms that can efficiently handle
\emph{inexact evaluations} of $\mathcal{J}$ and its derivatives.

An algorithm that satisfies these two key requirements is a trust-region method
termed `Lin-Mor\'{e}' in ROL.
This method, developed in~\cite[Algorithm 1]{kouri2022matrix}, is a generalization of the
trust-region Newton method proposed by Lin and Mor\'e~\cite{lin1999newton} and includes
the mechanisms to handle inexact function evaluations, developed recently
by Baraldi and Kouri~\cite{baraldi2023proximal}.
At each iteration, the method solves the subproblem
\begin{equation}
\label{eq:TR}
\minimize_{\tilde g \in \R^n} \;\; m_k(\tilde g)
\qquad \text{subject to}  \qquad \tilde g \in \mathcal{C},
                          \quad \|\tilde g - {\tilde g}_k\| \le \Delta_k,
\end{equation}
where $m_k : \R^n \rightarrow \R$ is a model of the objective function $\mathcal{J}$,
${\tilde g}_k$ is the current ($k$-th) iterate, and
$\Delta_k > 0$ is the trust-region radius.
The model we use is the second-order Taylor expansion of $\mathcal{J}$ around ${\tilde g}_k$,
\begin{equation}
\label{eq:TRmodel}
  m_k(\tilde g) = \langle \nabla^2 \mathcal{J}({\tilde g}_k) (\tilde g - {\tilde g}_k), \tilde g - {\tilde g}_k \rangle
  + \langle \nabla \mathcal{J}({\tilde g}_k) , \tilde g - {\tilde g}_k \rangle
  + \mathcal{J}({\tilde g}_k),
\end{equation}
which requires the derivative computations~\eqref{eq:func-L2} or~\eqref{eq:func-semi},
discussed in Section~\ref{sec:scalable}.

The constraint $\tilde g \in \mathcal{C}$ is enforced through projections.
Recalling the distance~$\mathcal{D}$ of a point $\tilde g$ to the set~$\mathcal{C}$,
\[
\mathcal{D}(\tilde g, \mathcal{C}) \coloneqq
\inf \left\{ \|\tilde g - y\| \; \vert \; y \in \mathcal{C} \right\},
\]
we note that for a closed convex set we can define the
projection~$P_{\mathcal{C}}(\tilde g) \in \mathcal{C}$,
which is a unique point satisfying
$\|\tilde g - P_{\mathcal{C}}(\tilde g)\| = \mathcal{D}(\tilde g, \mathcal{C})$.
Alternatively, we can write
\[
  P_{\mathcal{C}}(\tilde g) = \argmin_{y \in \mathcal{C}} \| \tilde g - y \|^2.
\]
The feasible set in problem~\eqref{eq:ocp-scaled} is a convex polyhedron,
\begin{equation}
\label{eq:proj}
C \coloneqq
\left\{ \tilde g \in \mathbb{R}^N \; \vert \;
a \leq \tilde g \leq b ~\wedge ~ \mathbbm{1}^{\top} D^{-1} \tilde g = 0 \right\},
\end{equation}
defined as the intersection of a box inequality and a single linear equality.
Projections onto $C$ are performed in ROL using
the approach by Dai and Fletcher \cite[Algorithm 1]{dai2006new}, which relies on simple bracketing
and a secant iteration.  Even though this is an iterative approach, it typically
converges in only a handful of iterations to machine zero
\red{(see \cite[Sec.~3]{bochev2013fast})}.
For problem~\eqref{eq:ocp-scaled}, the norm in~\eqref{eq:TR} and the scalar
product in~\eqref{eq:TRmodel} are defined consistently so
that $\|\cdot\| = \sqrt{\langle \cdot,\cdot \rangle} = \|\cdot\|_{M_L}$.

Next, we consider the solution of~\eqref{eq:TR} with the quadratic model~\eqref{eq:TRmodel}.
The `Lin-Mor\'{e}' method in ROL applies a truncated CG method~\cite{conn2000trust}, which
in the case of~\eqref{eq:proj} iterates in the null space of the equality
constraint~$\mathbbm{1}^{\top} D^{-1} \tilde g = 0$
while also satisfying the trust-region constraint.
As the CG steps may violate the bounds in~\eqref{eq:proj}, the truncated
CG iteration is followed by a projected search to restore feasibility, as necessary.
In general, this also means that all iterates remain feasible, up to the accuracy of
the projection onto the convex set $\mathcal{C}$.

We noted that the applications of the solution operator~$\tilde S$
in the evaluations of the objective function and
its derivatives are performed using a CG method with a multigrid preconditioner.
As a consequence, the objective function, the gradient, the
application of the Hessian, and our model~\eqref{eq:TRmodel} are all computed approximately.
To handle such approximations in a rigorous and efficient manner, the `Lin-Mor\'{e}' method
in ROL includes special inexactness conditions, due to~\cite{baraldi2023proximal}.
These conditions provide tolerances for the inexact evaluations, which
are easily implemented as \emph{relative stopping tolerances} for the
CG solver that applies the solution operator~$\tilde S$.
To control the inexactness in the objective function value, ROL
uses~\cite[Algorithm 3]{baraldi2023proximal}.
For the inexactness in the gradient computation, ROL
uses~\cite[Algorithm 4]{baraldi2023proximal}.
For the application of the Hessian, it is sufficient to
select any reasonable, e.g., fixed, relative stopping tolerance, such as 0.1.

The termination criteria for the `Lin-Mor\'{e}' method are based on the projected
gradient criticality measure
\[
  \left\| P_{\mathcal{C}} \left( \tilde g - \nabla \mathcal{J}(\tilde g) \right) - \tilde g \right\|.
\]
Unless noted otherwise, we stop the optimization algorithm at ${\tilde g}_\text{final}$
if the criticality measure is reduced by at least six orders of magnitude relative to
its value at the initial guess, i.e.,
\[
 \left\| P_{\mathcal{C}} \left( {\tilde g}_{\text{final}} -
    \nabla \mathcal{J}({\tilde g}_{\text{final}}) \right) - {\tilde g}_{\text{final}} \right\| /
 \left\| P_{\mathcal{C}} \left( {\tilde g}_0 - \nabla \mathcal{J}({\tilde g}_0) \right) - {\tilde g}_0 \right\|
    \le 10^{-6} ,
\]
or if the criticality measure is very small, specifically
\[
 \left\| P_{\mathcal{C}} \left( {\tilde g}_{\text{final}} -
    \nabla \mathcal{J}({\tilde g}_{\text{final}}) \right) - {\tilde g}_{\text{final}} \right\|
    \le 10^{-9} .
\]

\section{Numerical experiments}
\label{sec:examples}
\PBB{This section examines numerically the new potential-target OB formulation in two distinct simulation scenarios that require the enforcement of discrete maximum principles to avoid spurious oscillations. The first one involves conservative projections (remapping) of nonsmooth data, while the second one applies the approach to} finite element discretizations of conservation laws. 

The result of our remapping test provides initial data for a solid body rotation test in which we solve a linear advection problem.  To show the benefits of optimal control in the context of nonlinear problems, we apply our OB algorithm to several benchmarks for the Cahn-Hilliard equation. In all numerical experiments, we use the diagonal scaling operator $D=M_L^{-1}$. In addition to checking the validity of imposed constraints and visually inspecting the quality of numerical solutions, we \PBB{evaluate} the efficiency and scalability of the proposed algorithms. 

\subsection{Remap problem} \label{test:rm}
The initial data for the advection
 test proposed by Le\-Veque \cite{leveque1996high} is given by 
$$
   u_0(x,y)=\begin{cases}
    u_0^{\rm hump}(x,y)
     &\text{if}\  \sqrt{(x - 0.25)^2 + (y - 0.5)^2}\le 0.15, \\
    u_0^{\rm cone}(x,y)
     &\text{if}\ \sqrt{(x - 0.5)^2 + (y - 0.25)^2}\le 0.15, \\
     1 &\text{if}\ \begin{cases}
       \left(\sqrt{(x - 0.5)^2 + (y - 0.75)^2}\le 0.15 \right) \wedge \\
       \left(|x - 0.5| \ge 0.025 \vee \ y\ge 0.85\right),
     \end{cases}\\
      0 &  \text{otherwise}.
    \end{cases}
   $$
   A good algorithm for projecting $u_0$ into the finite element space should be
   accurate, conservative and bound preserving.
   We use optimal control to enforce these principles.
   
   Let $b \coloneqq \left(\int_\Omega \varphi_i u_0 \dx\right)_{i=1}^N$  be the
   right-hand side of the linear systems
   $$ M_C u^H = b\quad{\rm and}\quad M_L u^L = b.$$
   The consistent $L^2$ projection
   $u_h^H=\sum_{j=1}^Nu_j^H\varphi_j$
   of $u_0$ is conservative and minimizes the $L^2$ error. However, $u_h^H$ is
   generally not bound preserving. The lumped-mass approximation
   $u_h^L=\sum_{j=1}^Nu_j^L\varphi_j$ is conservative and bound
   preserving
   \cite{kuzmin-hajduk2022,kuzmin2010failsafe}. However, $u_h^L$ is less accurate than $u_h^H$ in smooth regions.

   We treat $u^H$ as a target and $u^L$ as a backup for
   our optimization procedure. By definition, the two conservative projections are related by
   \beq\label{L2target}
   M_L u^H = M_L u^L + \left(M_L - M_C\right) u^H.
   \eeq
   This representation of the unconstrained remapping scheme has the structure of  an explicit Euler update with the pseudo-time step $\Delta t = 1$.

Introducing a generic vector of flux potentials $\dot u$, we generalize
   \eqref{L2target} as follows:
   $$ M_L u = M_L u^L + \left(M_L - M_C\right) \dot u.$$
   The choices $\dot u=u^H$ and $\dot u=0$ produce $u=u^H$ and
    $u=u^L$, respectively.
The optimal potential $\dot u$ is the solution of
\begin{align*}
  \minimize_{\tilde g \in \mathbb{R}^N} &\;\;
    \frac{1}{2} \left\| \tilde S \tilde g - u^H \right\|^2_{M_L-M_C} \\
  \text{subject to} &\;\;
  \mathbbm{1}^{\top}D^{-1}\tilde g = 0, \\
  &\;\;
  D M_L\left( u^{\min} - u^L\right) \leq \tilde g \leq D M_L \left( u^{\max} - u^L\right).
\end{align*}
The local bounds that we use in this remapping test are defined by
$$ u^{\min}_i = \min_{j\in \mathcal N_i} u_j^L,\qquad u^{\max}_i = \max_{j\in \mathcal N_i} u_j^L,$$
where $\mathcal N_i$ is the set node indices $j$ such that $m_{ij}\ne 0$.
Since the low-order backup~$u^L$ is bound preserving, we have the property that
$$
u_i^{\min}\ge u^{\min}=\min_{\Omega} u_0,\qquad
u_i^{\max}\le u^{\max}=\max_{\Omega} u_0.
$$

The discontinuity of $u_0$ results in a violation of local and
global bounds by $u_h^H$
in this test. The gradients of continuous finite element approximations become steeper
as the mesh is refined. As a consequence, solvers for optimization 
problems require more iterations to converge. In our experience, the
use of the semi-norm \mbox{$\|\cdot\|_{M_L-M_C}$} in our definition of the
objective function is essential for achieving scalable convergence
behavior and penalizing the $H^1$ semi-norm of the projection error. The imposition
of local bounds prevents oscillations within the range $[u^{\min},u^{\max}]$
but small distances between $u_i^{\min}$ and  $u_i^{\max}$ require
more projection steps compared to box constraints with
global bounds.

We discretize the computational domain $\Omega=(0,1)^2$ using a uniform 
grid that consists of 64$\times$64 square cells.
Projections of $u_0$ are performed into the space of continuous bilinear
($\mathbb{Q}_1$) finite elements. The backup solution $u_h^L$ and target solution $u_h^H$ are shown in Figures \ref{fig:remap_backup} and \ref{fig:remap_target}, respectively. At first, we notice that the nodal values of $u_h^L$ stay in $[u^{\min},u^{\max}]=\left[0,1\right]$, while $u_h^H$ exhibits overshoots and undershoots. The optimal control solution $u_h^{\rm opt}$ shown in Figure~\ref{fig:remap_opt} preserves the global bounds as accurately as machine precision permits. The difference between $u_h^{\rm opt}$ and the nonconservative interpolant $u_h^{\rm int}=\sum_{j=1}^Nu_0(\mathbf{x}_j)\varphi_j$ of the input data is displayed in Figure~\ref{fig:remap_diff}. We observe that the two approximations are comparable in terms of accuracy. The largest deviations occur around the slotted cylinder, as well as at the top and bottom of the cone.

The constraints of our optimization problem are satisfied by the converged solution~$u_h^{\rm opt}$. The stopping criteria for the calculation of $u_h^{\rm opt}$ are met after 15 iterations in this example. For comparison purposes, we present the intermediate solutions $u^{\text{opt}}_3$ and $u^{\text{opt}}_7$ after three and seven iterations in Figures \ref{fig:remap_opt_i1} and \ref{fig:remap_opt_i2}, respectively. Although these approximations are not yet optimal, the bounds of the inequality constraints are preserved, and the mass defect $\sum_{j=1}^Ng_j$ is zero to machine precision. The quality of  constrained projections improves in the process of optimization. A visual comparison of $u^{\text{opt}}_{h,3}$ and $u^{\text{opt}}_{h,7}$ reveals some differences in the resolution of the upper part of the cylinder. We see that $u^{\text{opt}}_{h,3}$ still exhibits some spurious oscillations, whereas $u^{\text{opt}}_{h,7}$ is free of them. After 12 optimization steps, no noticeable differences are observed between subsequent iterates.
\begin{figure}
    \begin{subfigure}{0.495\textwidth}
        \includegraphics[width=\textwidth]{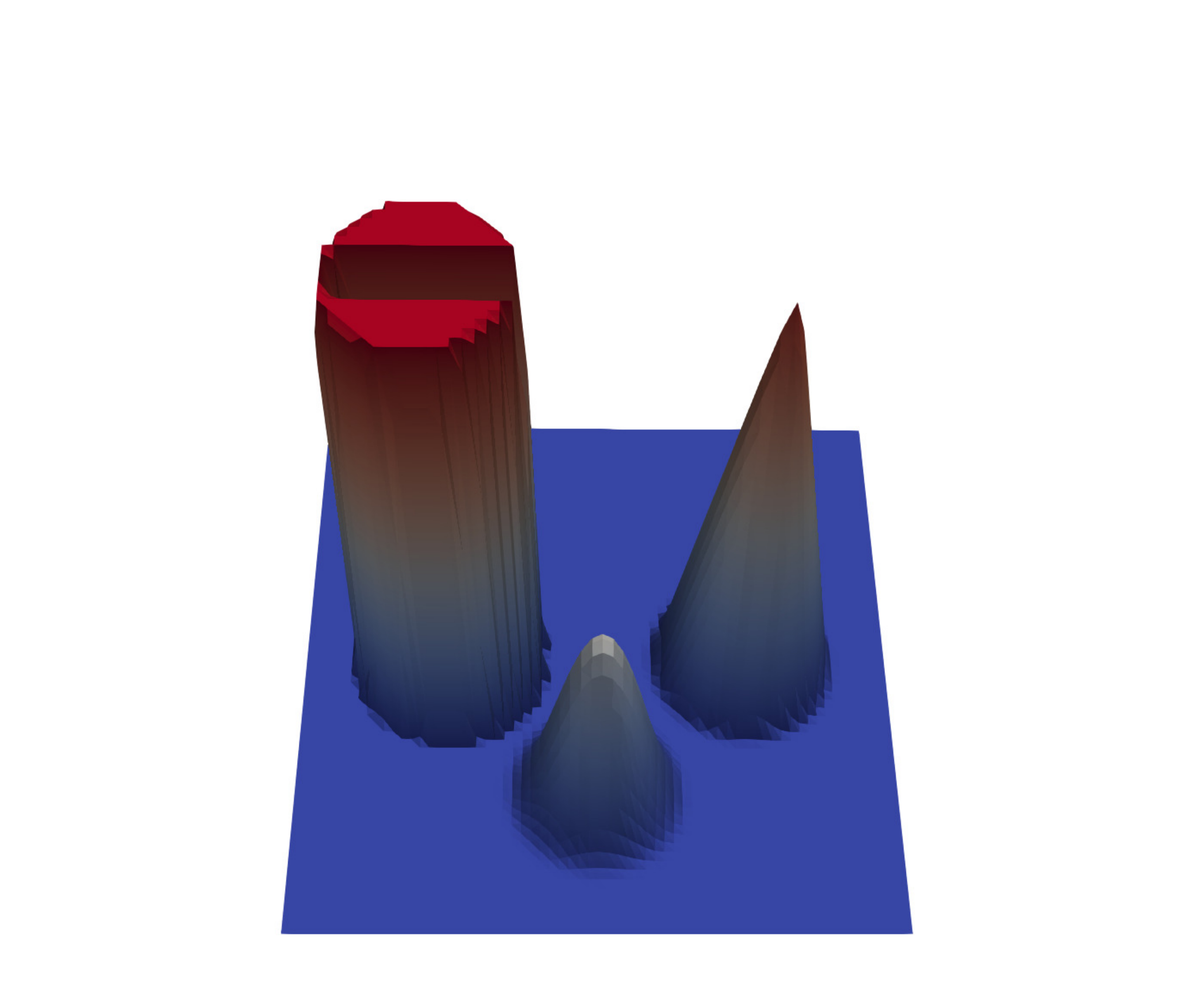}
        \caption{Backup solution;\\ $u_h^L \in \left[0,1\right]$.}
        \label{fig:remap_backup}
    \end{subfigure}
    \begin{subfigure}{0.495\textwidth}
        \includegraphics[width=\textwidth]{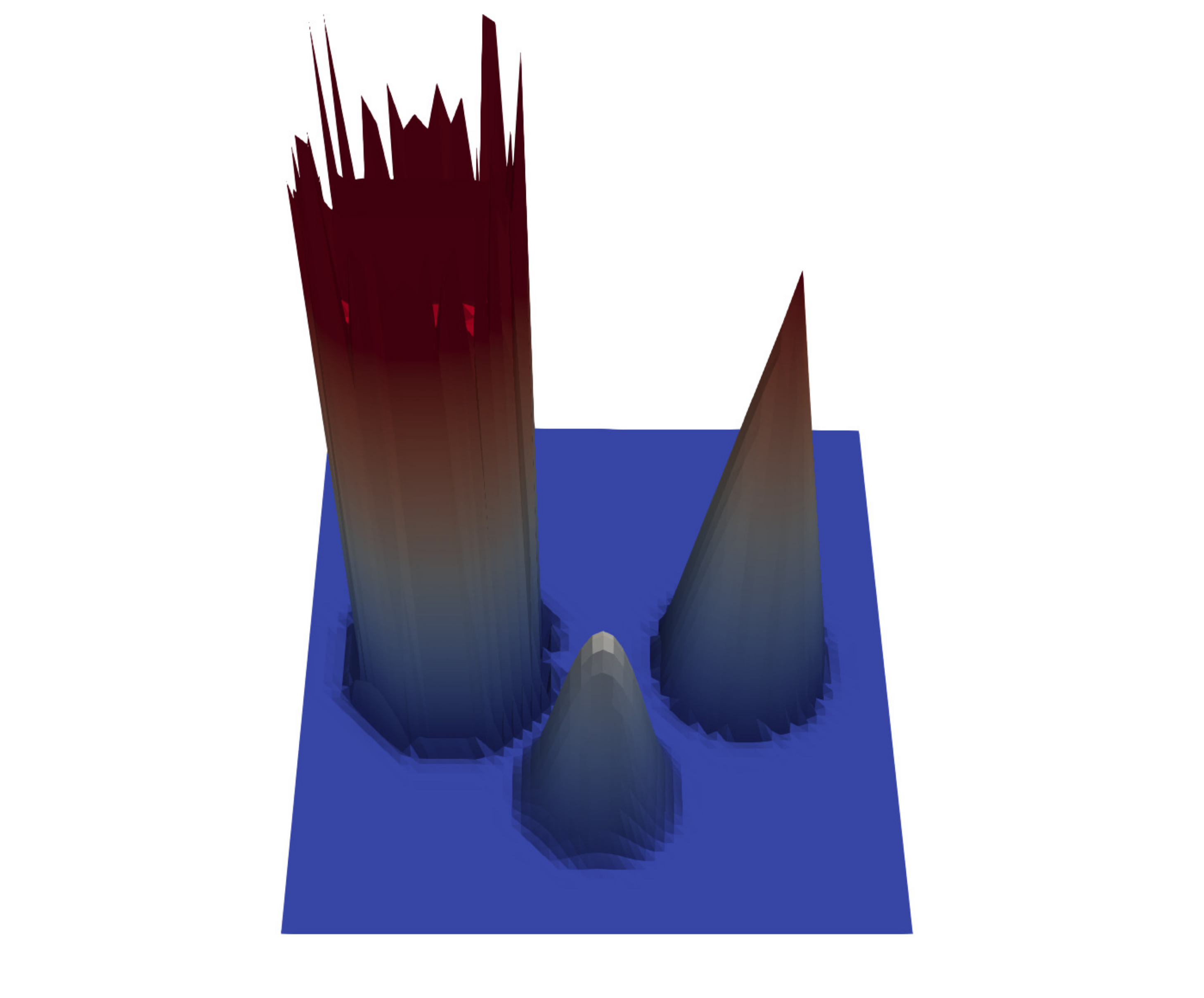}
        \caption{Target solution;\\ $u_h^H \in \left[-0.334,1.352\right]$.}
        \label{fig:remap_target}
    \end{subfigure}
    \begin{subfigure}{0.495\textwidth}
        \includegraphics[width=\textwidth]{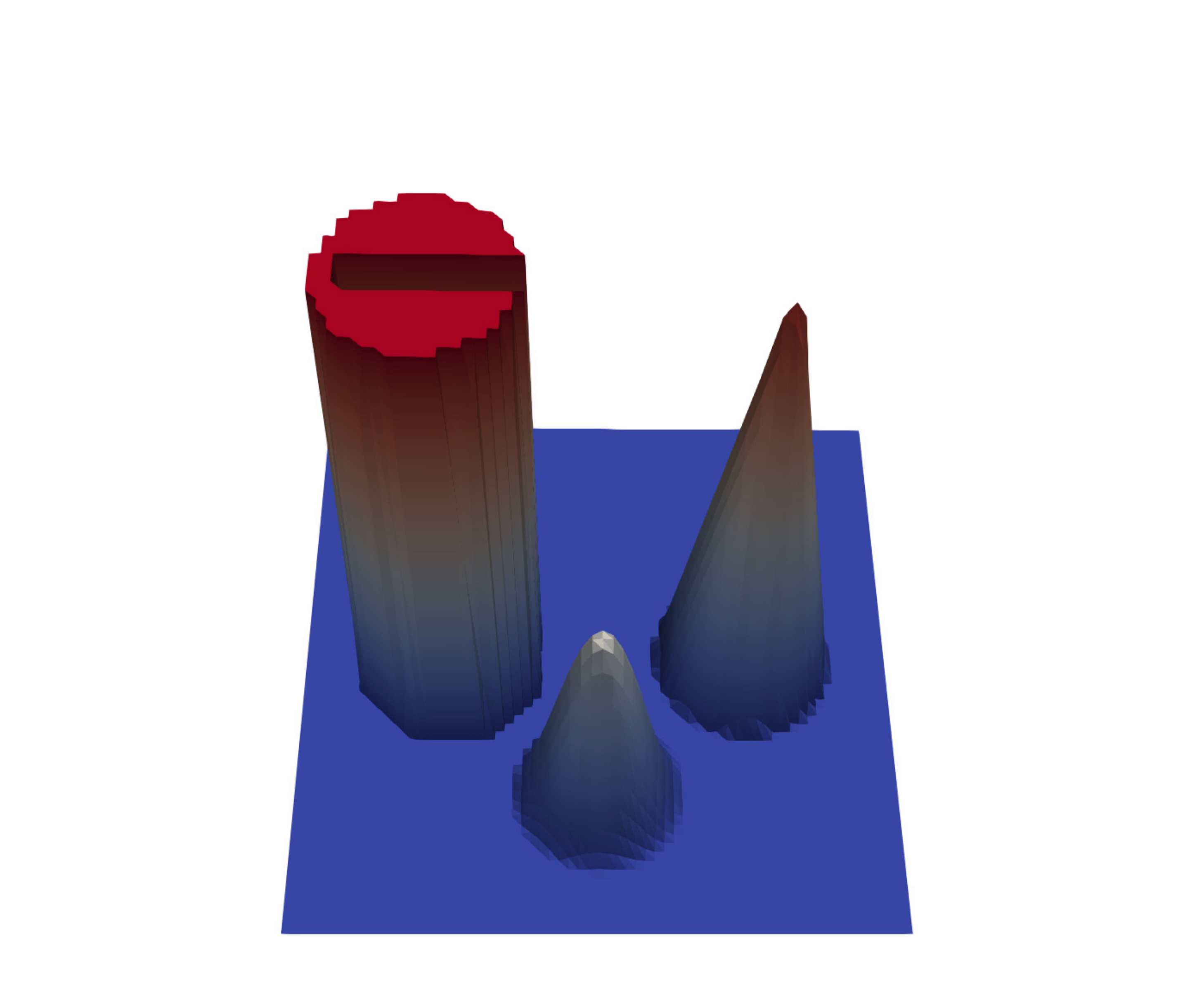}
        \caption{Optimal solution;\\ $u_h^{\text{opt}} \in \left[0,1\right]$.}
        \label{fig:remap_opt}
    \end{subfigure}
    \begin{subfigure}{0.495\textwidth}
        \includegraphics[width=\textwidth]{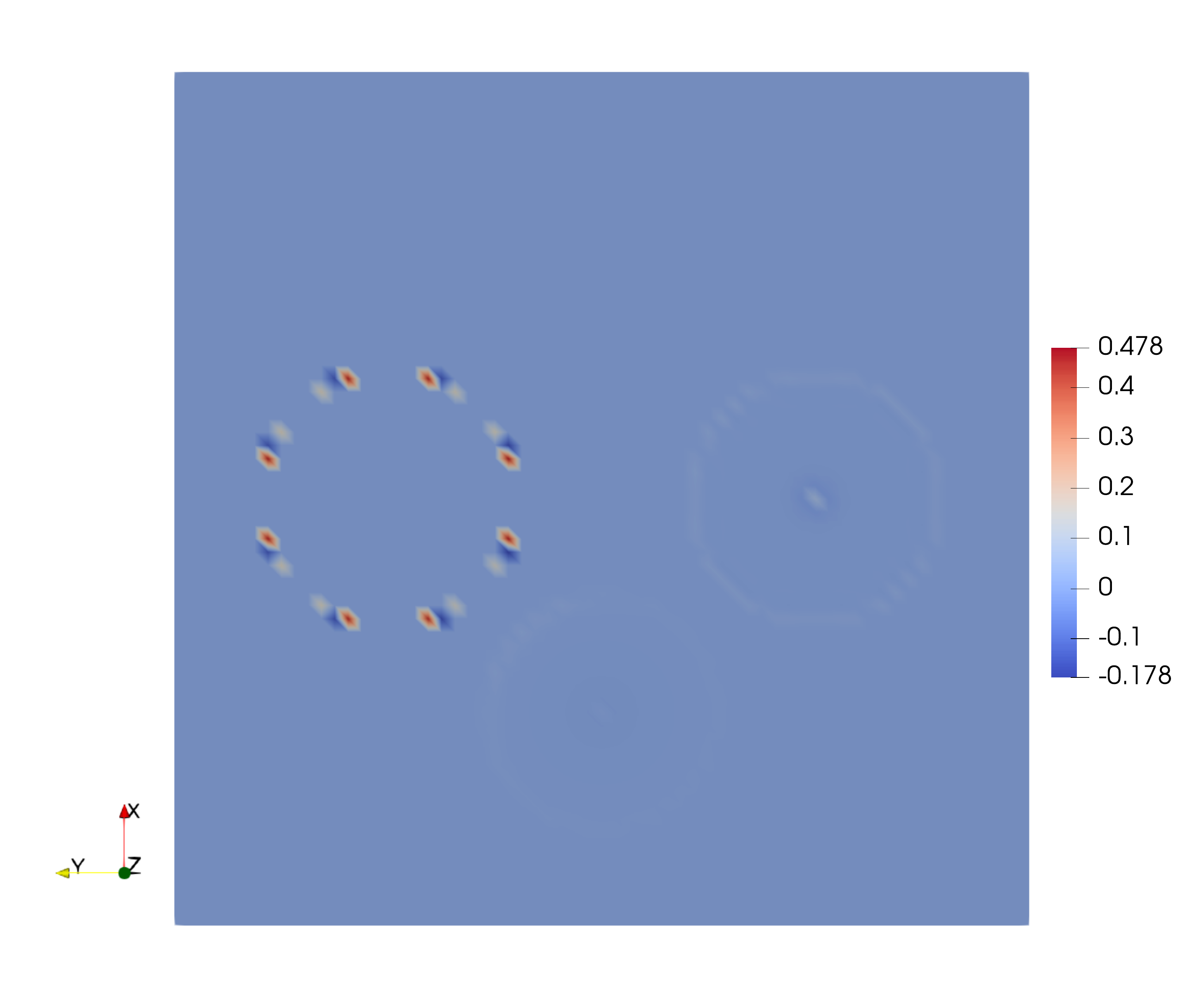}
        \caption{Optimization vs.\ interpolation;\\
        \centering $u_h^{\rm int} - u_h^{\text{opt}} \in \left[-0.178, 0.478 \right]$.
        }
        \label{fig:remap_diff}
    \end{subfigure}
    \begin{subfigure}{0.495\textwidth}
        \includegraphics[width=\textwidth]{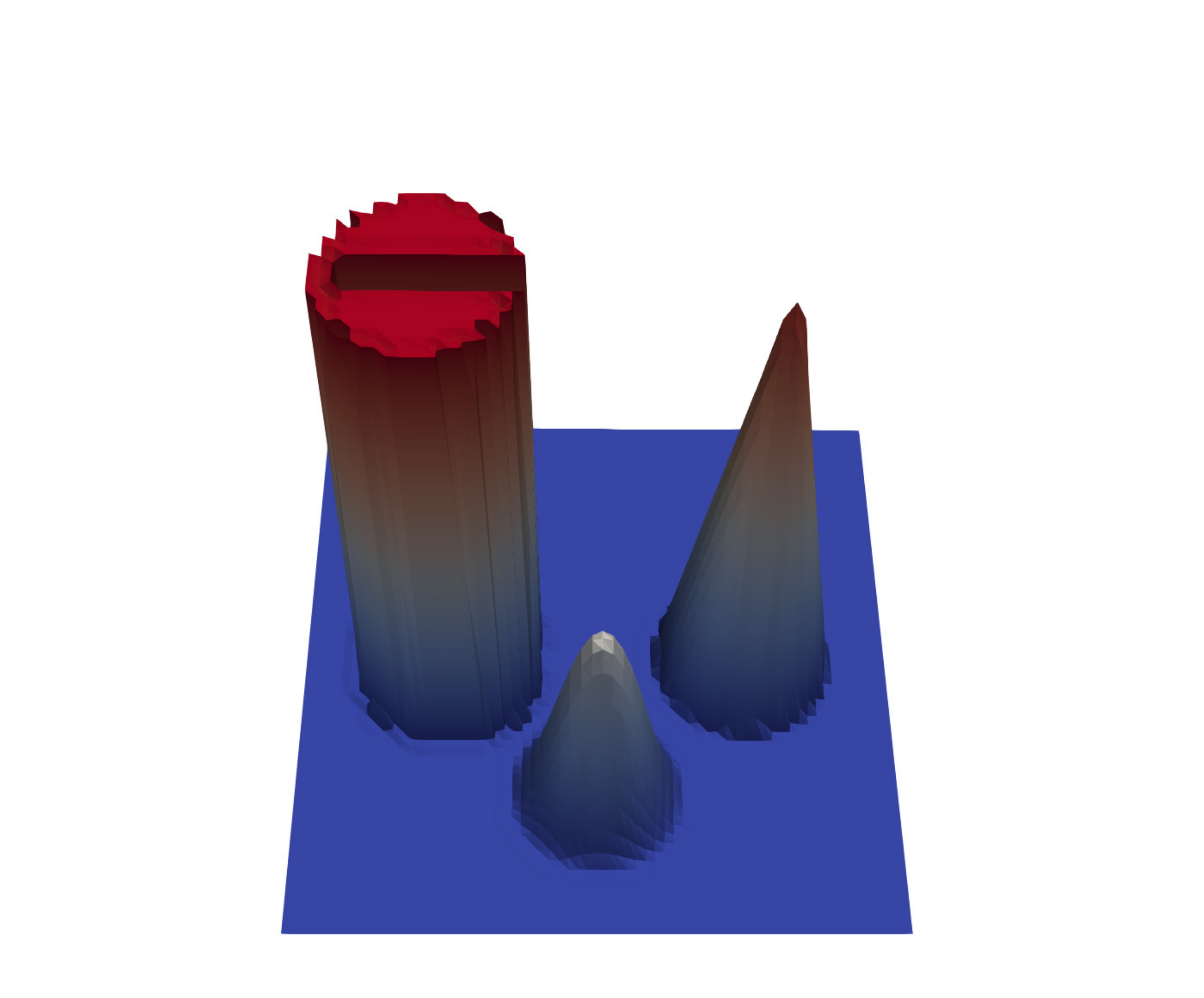}
        \caption{Intermediate result after 3 iterations;\\ $u^{\text{opt}}_{h,3} \in \left[0,1\right]$.}
        \label{fig:remap_opt_i1}
    \end{subfigure}
    \begin{subfigure}{0.495\textwidth}
        \includegraphics[width=\textwidth]{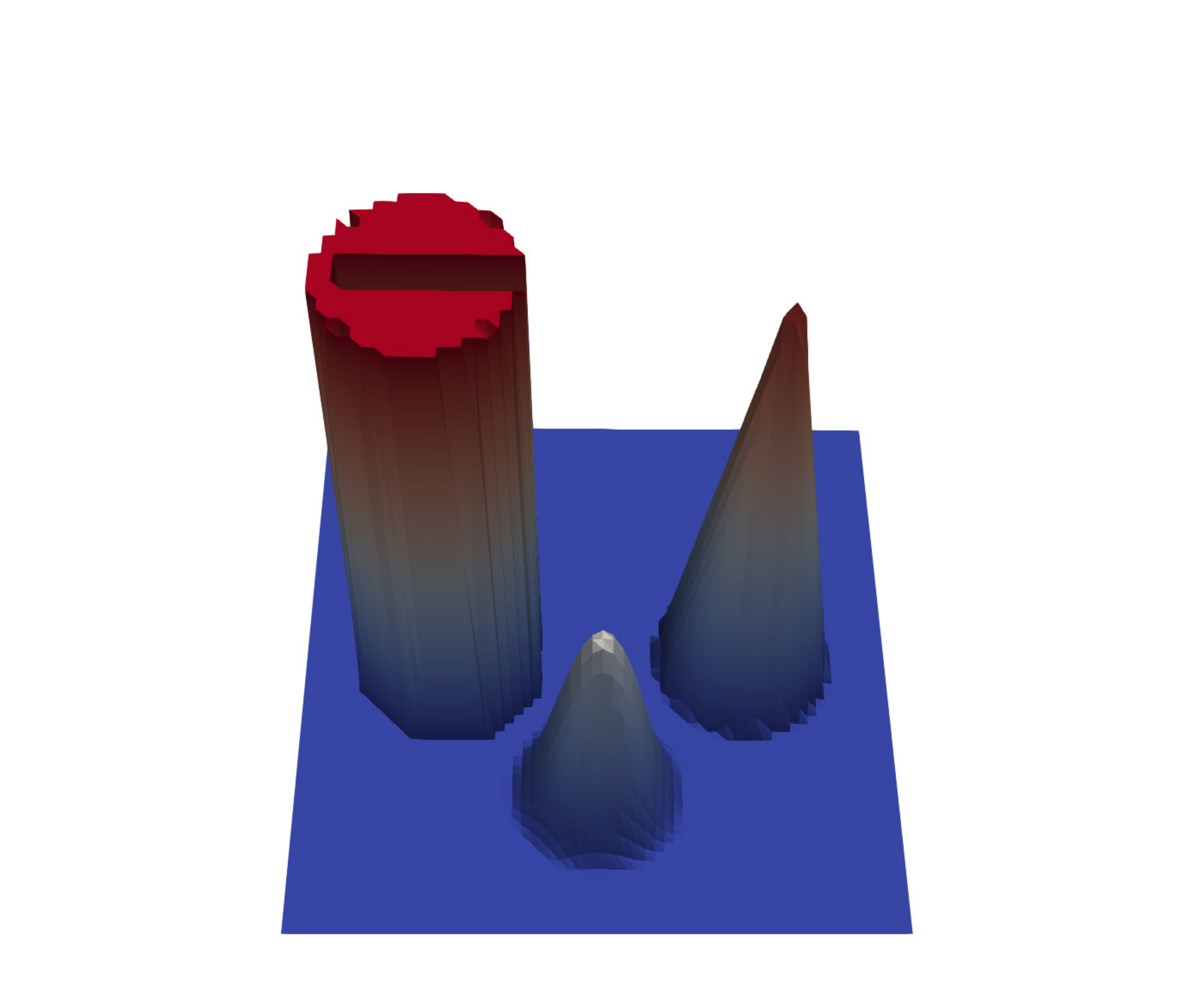}
        \caption{Intermediate result after 7 iterations;\\ $u^{\text{opt}}_{h,7} \in \left[0,1\right]$.}
        \label{fig:remap_opt_i2}
    \end{subfigure}
    \caption{Approximate solutions of the remap problem.}
\end{figure}

\begin{remark}
  In practice, we never use $\dot u=Sg$ to calculate the matrix-vector product $(M_L-M_C)\dot u$ in the final or intermediate stages of the optimization procedure. Since $(M_L-M_C)Sg=g$ by definition of the solution operator $S$, the direct use of $g$ is more efficient. Moreover, the calculation of $(M_L-M_C)Sg$ produces numerical errors that may result in slight violations of the imposed constraints. However, an approximate solution to $S\dot u=g$ is needed to evaluate the objective function and the gradient. In this context, errors due to imprecise calculation of $\dot u$ have no adverse impact on the validity of the constraints.
\end{remark}

To verify the scalability of our proposed algorithm, we solve another remap problem on a sequence uniformly refined meshes and check how the total number of iterations depends on the mesh size. The data to be projected is given by
$$u_0(x,y) = H_\epsilon (\phi(x))H_\epsilon (\phi(y)),$$
where we use $\phi(x) = 0.1 - \left | x - 0.5\right |$ and the regularized Heaviside function
$$
	H_\epsilon\left(\phi\right) = \begin{cases}
	0, &\text{if} ~\phi \leq -\epsilon, \\
	\frac12 \left(1+\frac\phi\epsilon + \frac1\pi\sin\left(\pi \frac\phi\epsilon\right)\right), &\text{if} ~-\epsilon < \phi < \epsilon,\\
	1, &\text{if} ~\phi \geq \epsilon,
	\end{cases}
$$
to generate a rectangular block with steep gradients for $\epsilon = \frac{3}{80}$. The boundedness of $\nabla u_0$ allows a better comparison of the iteration numbers for different mesh levels.

        Figure \ref{fig:conv_study_init} shows the result of the optimization-based projection into the finite element space corresponding to the $64\times64$ grid. The results of our grid convergence study are listed in Table \ref{tab:scal_conv}.
\begin{table}\centering
\begin{tabular}{ccccccc}
\hline
mesh size && \#iter & \#fval & \#grad & \#proj & $\chi(\tilde g_\text{final})$ \\ 
\hline
1 / 16 && 8 & 9 & 9 & 71 & 3.86e-12 \\
1 / 32 && 7 & 8 & 8 & 55 & 8.64e-11 \\
1 / 64 && 5 & 6 & 6 & 41 & 4.71e-12 \\
1 / 128 && 6 & 7 & 7 & 48 & 1.12e-11 \\
1 / 256 && 1 & 2 & 2 & 13 & 5.89e-11 \\
1 / 512 && 1 & 2 & 2 & 13 & 5.86e-11 \\
1 / 1024 && 1 & 2 & 2 & 15 & 1.12e-12 \\
\hline
\end{tabular}
\caption{Impact of mesh refinement on the scalability of optimal control.}
\label{tab:scal_conv}
\end{table}
We terminate the optimization algorithm at the iterate~$\tilde{g}_{\text{final}}$
when the stopping criterion
\[
  \chi(\tilde g_\text{final}) \coloneqq
  \min(\|P_C(\tilde{g}_{\text{final}} - \nabla \tilde J(\tilde{g}_{\text{final}})) - \tilde{g}_{\text{final}} \|,
  \tilde J(\tilde{g}_{\text{final}})) < 10^{-10} 
  \]
is met. In other words, we stop
iterating if the criticality measure or the value of the objective function
become very small. In either case, the constraints are satisfied to machine precision.
In our experiments, the number iterations (\#iter) becomes equal to 1
on fine meshes.
The numbers of function evaluations (\#fval) and gradient computations (\#grad)
do not change significantly as we refine the mesh.
The numbers of projections onto the convex set $C$ (\#proj) vary slightly. However, there is no appreciable increase as the mesh size decreases.
Hence, our optimal control approach scales well with mesh refinement.

\begin{figure}
    \centering
    \includegraphics[width=.5\textwidth]{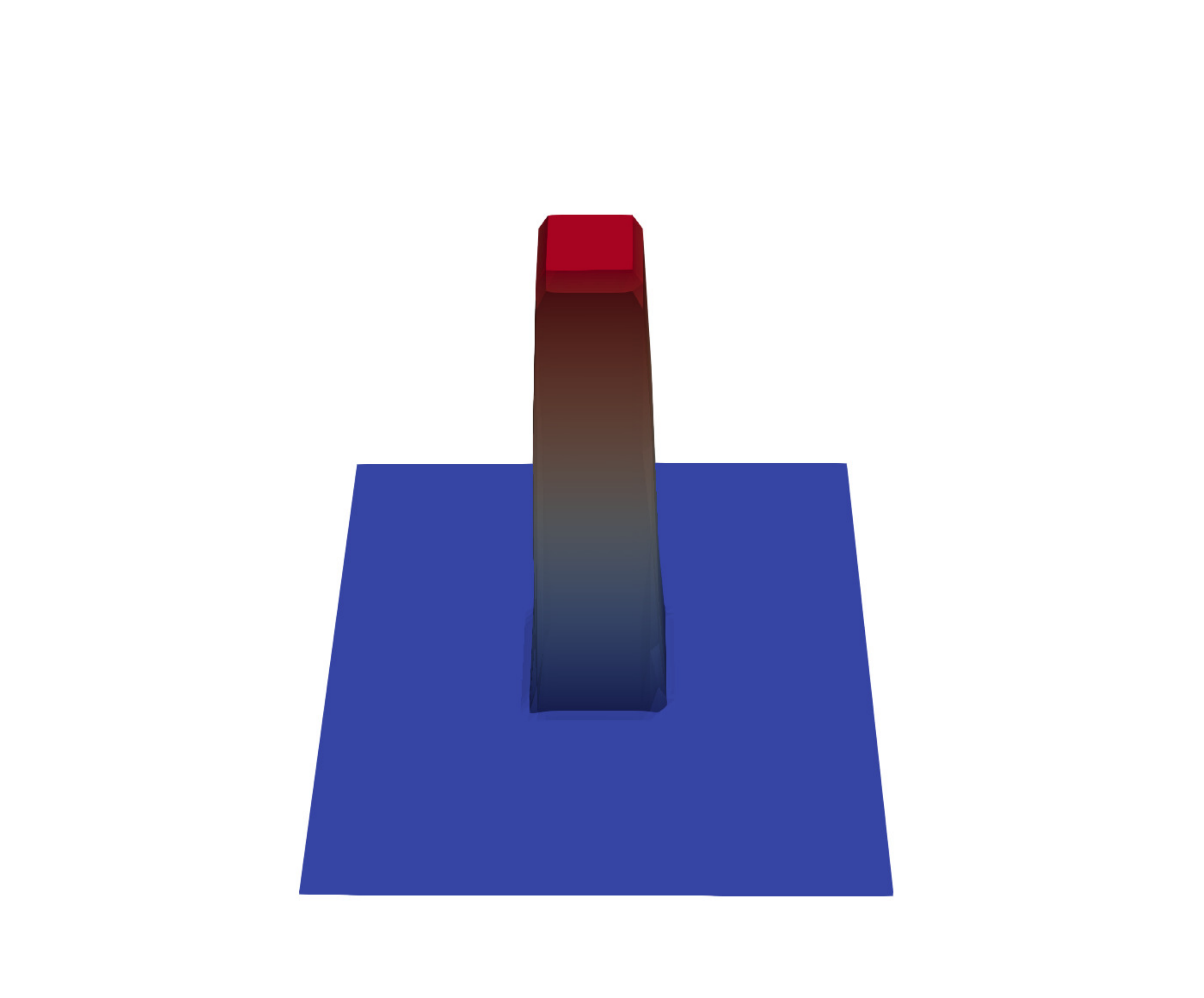}
    \caption{Optimization-based projection $u_h^{\rm opt} \in \left[0,1\right]$ on the 64 $\times$ 64 grid.}
    \label{fig:conv_study_init}
\end{figure}

\subsection{Advection problem}
In our second numerical example, we apply our optimal control formulation to a
finite element discretization of the initial-boundary value problem
\begin{alignat*}{4}
    \pd{u}{t} + \nabla\cdot\left(\mathbf{v} u\right) &= 0 & & \text{ in } & & \Omega\times\left(0,2\pi\right], \\
    u &= u_0 & & \text{ on } & & \Omega\times\left\{0\right\}, \\
    u &= 0 & & \text{ on } & & \Gamma_{-}\times\left[0,2\pi\right],
\end{alignat*}
where $\Gamma_{-}$ is the inflow boundary of $\Omega=(0,1)^2$. The
velocity field is given by $$\mathbf{v}\left(x,y\right) = \left( 0.5 - y, x - 0.5 \right)^{\top}.$$
The initial data $u_0$ is defined as in the remap test. For simplicity,
our numerical solutions are initialized by the lumped-mass $L^2$ projection
of $u_0$.

Using the standard Galerkin method to discretize the time-dependent 
advection equation in space, we obtain the semi-discrete problem
$$
  M_C \td{u}{t} =  K u+ b(u). 
  $$
The coefficients of $K=(k_{ij})_{i,j=1}^N$ and $b=(b_i)_{i=1}^N$ are given by  
\begin{gather*}
  ~ k_{ij} = \int_\Omega \nabla\varphi_i\cdot\mathbf{v} \varphi_j \dx, \qquad
    b_i(u) = - \frac12\int_{\Gamma} \varphi_i \left(\left | \mathbf{v} \cdot \mathbf{n} \right | + \mathbf{v} \cdot \mathbf{n} \right) u \ds.
\end{gather*}
For stabilization purposes, we include high-order dissipation
of the form \cite{kuzmin2023dissipation}
$$
d_h(w_h,u_h)=\nu\int_\Omega\nabla w_h\cdot (\nabla u_h-P_h\nabla u_h)\dx,
$$
where $P_h$ is the consistent $L^2$ projection operator and
$w_h$ is a test function belonging to the finite element space.
We use the constant artificial diffusion coefficient $\nu=\frac{\lambda h}2$,
where $\lambda=\frac{1}{10}
\max_{\mathbf{x}\in \bar\Omega}|\mathbf{v}(\mathbf{x})|$ and
$h$ is the mesh width parameter.
The bilinear form $d_h(\cdot,\cdot)$ defines the entries
$d_{ij}=d_h(\varphi_i,\varphi_j)$ of 
the high-order stabilization operator $D=(d_{ij})_{i,j=1}^N$ that
we add to $K$.

Let us discretize in time using Heun's two-stage strong stability preserving (SSP) Runge-Kutta method. The resulting fully discrete scheme reads
\begin{align*}
    u_{[1]}^H &= u_{[0]}^H + \Delta t M_C^{-1} (K + D) u_{[0]}^H, \\
    u_{[2]}^H &= u_{[1]}^H + \Delta t M_C^{-1} ( K + D) u_{[1]}^H, \\
    u_{[3]}^H &= \frac12 \left(u_{[0]}^H  + u_{[2]}^H \right).
\end{align*}
In each stage, we need to solve a linear discrete problem of the form
$$
M_Lu^T = M_Lu^B + \Delta t(K + D) u^B+(M_L-M_C)(u^T - u^B).
$$
Replacing $u^T - u^B$ with a general vector $\Delta t \dot u$ of flux
potentials, we constrain
\beq\label{AdvOB}
M_Lu = M_Lu^B + \Delta t(K + D) u^B+\Delta t (M_L-M_C)\dot u
\eeq
to preserve the local bounds
$$ u^{\min}_i = \min_{j\in \mathcal N_i} u_j^B,\qquad u^{\max}_i = \max_{j\in \mathcal N_i} u_j^B.$$
The objective function of our optimization problem is defined
using $$\dot u^T = \frac{u^T - u^B}{\Delta t}.$$ 
\begin{remark}
  If $u^L = u^B + \Delta t r^L$ is a consistent, conservative,
  and bound-preserving low-order approximation to $u^T$, then
  \eqref{AdvOB} is equivalent to
  $$
  M_Lu = M_Lu^L +\Delta t[(K + D) u^B-r^L]
  +(M_L-M_C)\dot v
  $$
  with $\dot v=\dot u+S(\Delta t r^L)$. The corresponding
  optimal solutions are the same but $u^L$ might be a better
  initial guess for the iterative solver than $u^B$.
\end{remark}

In our numerical experiments, we use $128\times128$ bilinear finite
elements and the time step $\Delta t = 10^{-3}$. Figure \ref{fig:adv_ana}
shows the (nonconservative) interpolant $u_h^{\rm int}=I_hu_0$ of the
initial condition. Numerical solutions are output
at the final time $T=2\pi$, at which the exact solution of the
solid body rotation problem coincides with the initial data. The
solution $u_h^H$ produced by the unconstrained target scheme is
depicted in Figure~\ref{fig:adv_ref}. Despite the use of high-order
stabilization, $u_h^H$ does not preserve the global
bounds $u^{\min}=0$ and $u^{\max}=1$. To better show the location of
undershoots and overshoots, we plot them in black in Figure
\ref{fig:adv_ref_vio}. To enforce local maximum principles,
we activate our optimal control procedure. The result is shown in
Figure~\ref{fig:adv_opt}. It can be seen that the constrained approximation
$u_h^{\rm opt}$ is locally bound preserving and free of spurious oscillations.
\begin{figure}
      \begin{subfigure}{0.495\textwidth}
        \includegraphics[width=\textwidth]{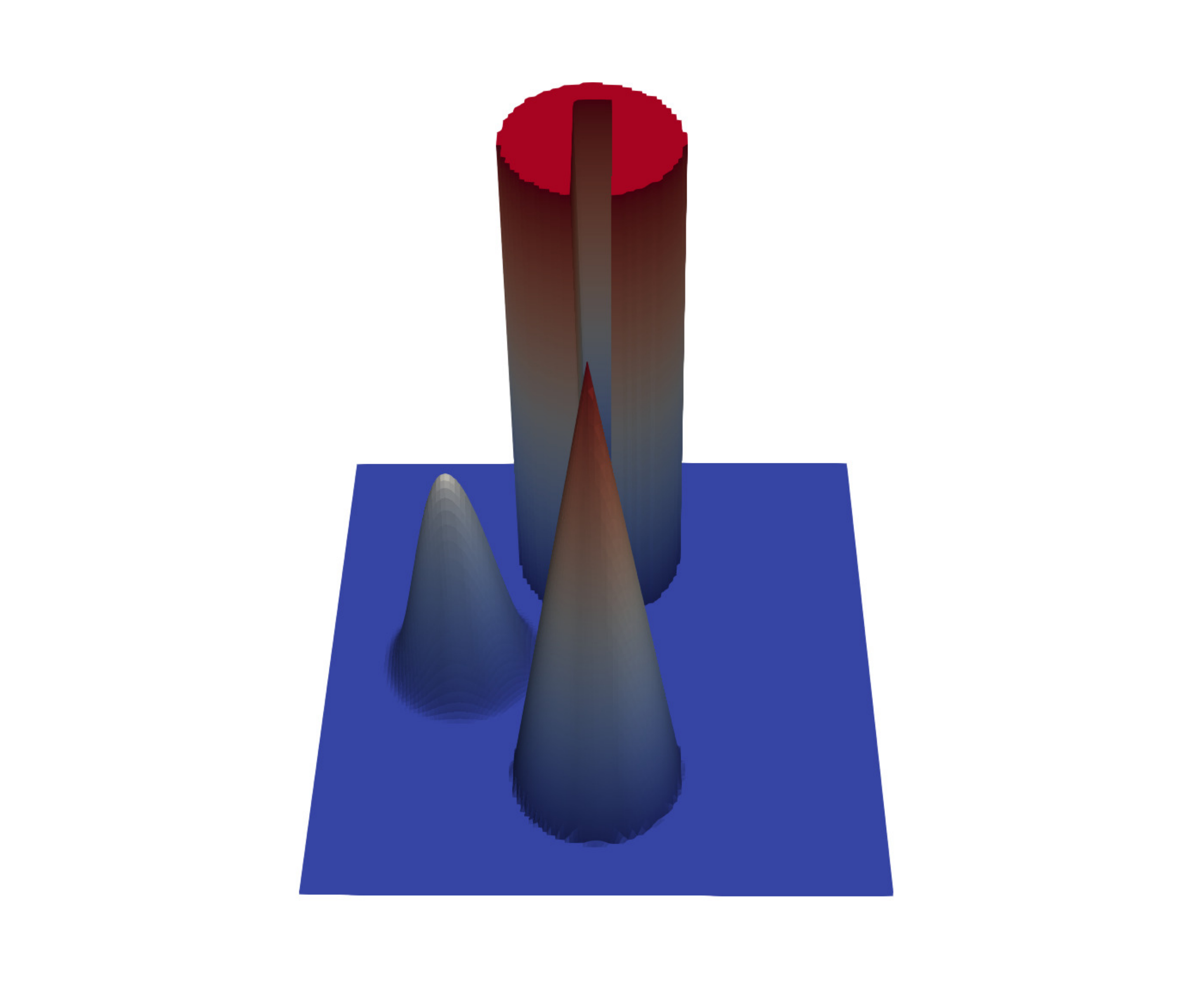}
        \caption{Interpolant of $u_0$; \\ $u_h^{\rm int} \in \left[0,1\right]$.}
        \label{fig:adv_ana}
    \end{subfigure}
    \begin{subfigure}{0.495\textwidth}
        \includegraphics[width=\textwidth]{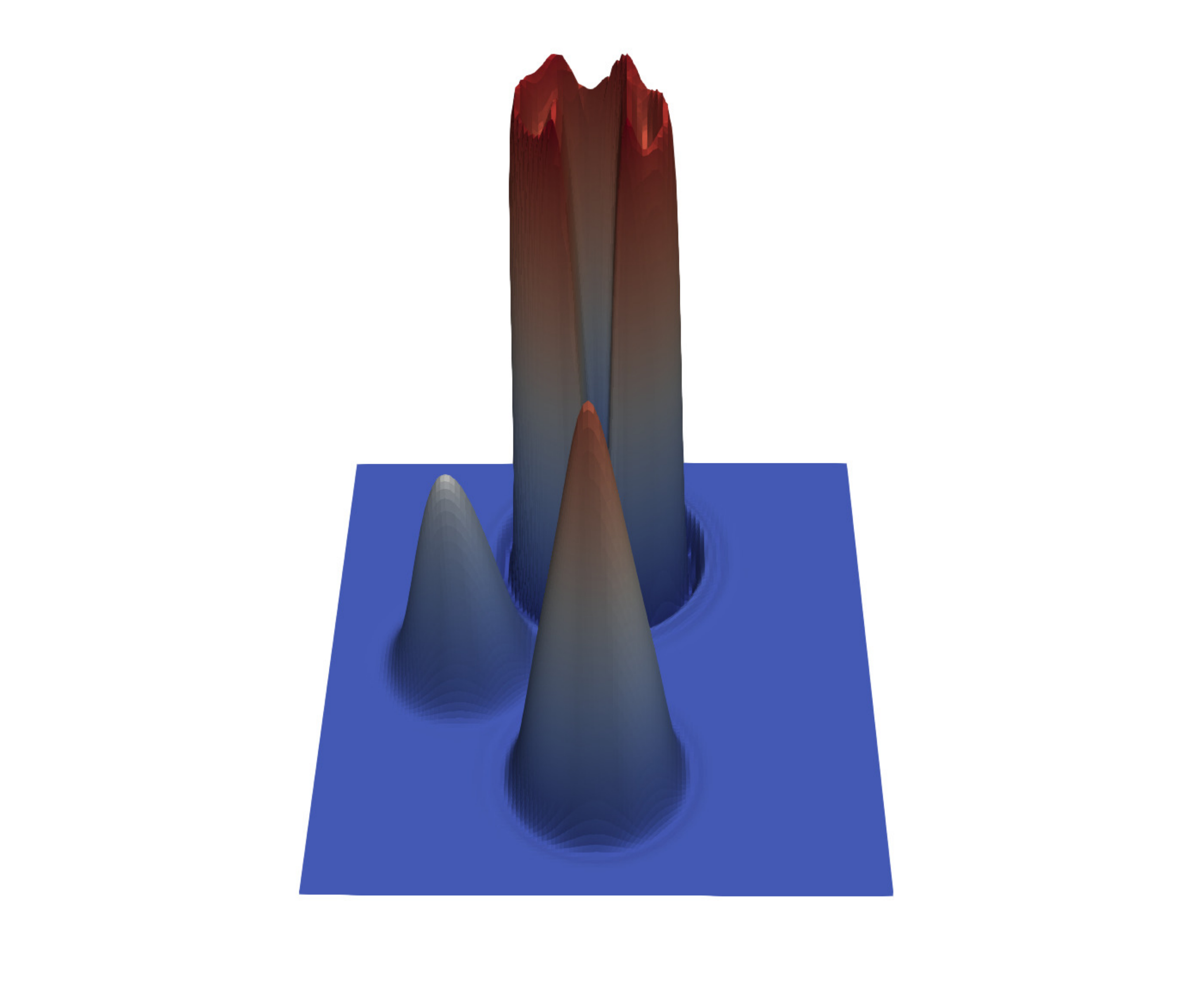}
        \caption{Target solution;\\ $u_h^H \in \left[-0.061,1.114\right]$.}
        \label{fig:adv_ref}
    \end{subfigure}
    \begin{subfigure}{0.495\textwidth}
        \includegraphics[width=\textwidth]{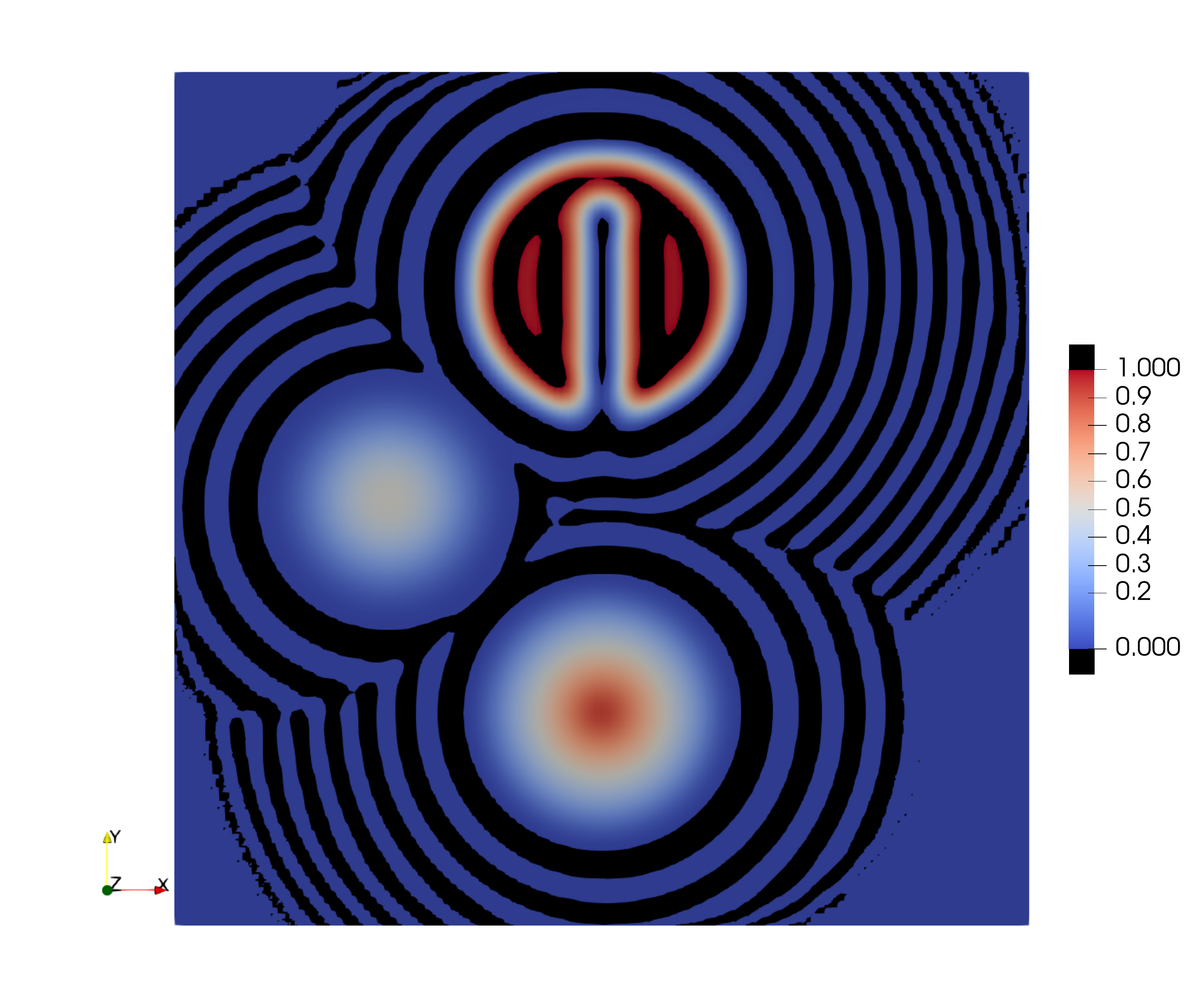}
        \caption{Undershoots/overshoots;\\ $u_h^H \in \left[-0.061,1.114\right]$.}
        \label{fig:adv_ref_vio}
    \end{subfigure}
    \begin{subfigure}{0.495\textwidth}
        \includegraphics[width=\textwidth]{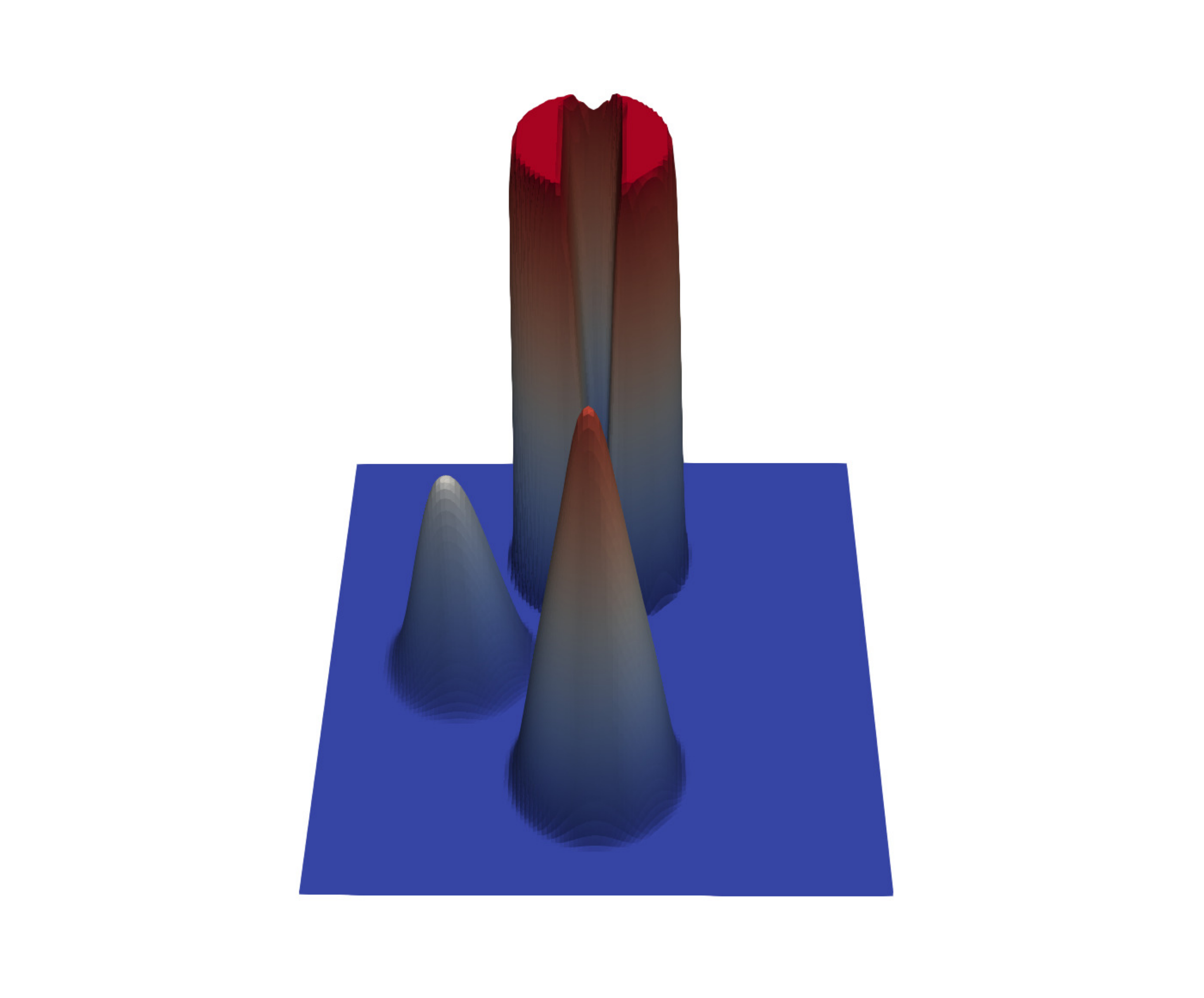}
        \caption{Optimal solution;\\ $u_h^{\text{opt}} \in \left[0,1\right]$.}
        \label{fig:adv_opt}
    \end{subfigure}
    \caption{Approximate solutions of the advection problem.}
\end{figure}

To assess the robustness of our optimal control approach and uncover additional computational speedups, we investigate how ROL's `Lin-Mor\'{e}' method performs for different combinations of \emph{fixed} relative stopping tolerances for MFEM's CG method, used to approximate the application of the solution operator $\tilde S$ in function, gradient, and Hessian evaluations. We label these CG tolerances fstop, gstop, and Hstop, respectively. The quantities of interest are listed in Table~\ref{tab:abb} and their values are presented in Table \ref{tab:scal_tol}. We find that inaccurate linear system solves may drastically reduce the overall computational cost without significantly affecting the accuracy of the result or the behavior of the optimization algorithm. In particular, rows five, six, and seven of Table \ref{tab:scal_tol} indicate that remarkable speedups can be achieved by using inaccurate approximations to Hessian-vector products in the optimization process. Moreover, we note that the mass error is near machine zero, and that the bound constraints are satisfied. Therefore, our optimal control approach enables a mixing of approximate scalable linear solvers, used to meet the accuracy criteria, with machine-accurate projections, used to enforce discrete maximum principles.
\begin{table}[h!]\centering
\begin{tabular}{c|c}
label &  quantity of interest \\
\hline
\#f & average number of function evaluations \\
\#g & average number of gradient evaluations \\
\#p & average number of projection steps \\
$\|\nabla\|$ & average norm of the final criticality measure \\
CPU & total CPU time  \\
mass err & average error in the mass constraint \\
\#V &  average number of multigrid V-cycles\\ & per Hessian-vector product
\end{tabular}
\caption{Meaning of column headings used in Table \ref{tab:scal_tol}.}
\label{tab:abb}
\end{table}

\begin{table}[h!]\centering
\begin{tabular}{ccccccccc}
\hline
fstop/gstop/Hstop && \#f & \#g & \#p & $\|\nabla\|$ & CPU & mass err & \#V \\ 
\hline 
1e-10/1e-10/1e-10 && \hphantom{0}23 & \hphantom{0}23 & 218 & 1.38e-07 & 18h19m & 7.37e-17 & 20.0 \\
1e-08/1e-08/1e-08 && \hphantom{0}23 & \hphantom{0}23 & 215 & 1.39e-07 & 14h47m & 6.83e-17 & 16.0 \\
1e-06/1e-06/1e-06 && \hphantom{0}23 & \hphantom{0}23 & 208 & 1.41e-07 & 11h33m & 7.38e-17 & 12.0 \\
1e-06/1e-06/1e-04 && \hphantom{0}23 & \hphantom{0}23 & 211 & 1.49e-07 & \hphantom{0}8h33m & 6.65e-17 & \hphantom{0}8.2 \\
1e-06/1e-06/1e-02 && \hphantom{0}25 & \hphantom{0}25 & 242 & 1.07e-06 & \hphantom{0}5h54m & 7.49e-17 & \hphantom{0}4.4 \\
1e-06/1e-04/1e-02 && \hphantom{0}24 & \hphantom{0}24 & 217 & 1.38e-06 & \hphantom{0}5h43m & 6.89e-17 & \hphantom{0}4.4 \\
1e-06/1e-06/1e-01 && \hphantom{0}31 & \hphantom{0}30 & 273 & 1.94e-05 & \hphantom{0}4h44m & 7.98e-17 & \hphantom{0}2.4 \\
1e-06/1e-06/0.250 && 136 & 114 & 982 & 8.70e-06 & 10h01m & 1.08e-16 & \hphantom{0}2.0\\
\hline
\end{tabular}
\caption{Solver test for the advection problem.
         Significant speedups, of nearly 4x, are observed when using inaccurate
         approximations of the Hessian-vector product.}
\label{tab:scal_tol}
\end{table}

As the final example in this subsection we investigate the \emph{mass spreading} behavior
of our optimal control approach.
We refer to \cite[Figure 21(d-f)]{peterson2024optimization} for comparisons with
the OBFEM-DD scheme and its variants, which satisfy local bounds and global mass
conservation but do not incorporate any notion of local conservation.
For this study, we use a $96$$\times$$96$ grid of bilinear finite elements
with periodic boundary conditions.
The velocity field is a translation, given by $\mathbf{v}(x,y)= (1,0)^\top$.
We set the final simulation time to $T=1$, which corresponds to a complete cycle
of the initial condition across the periodic boundaries $x=0$ and $x=1$.
Our result in Figure~\ref{fig:mass_spread} satisfies the bound constraints and
conserves the mass globally.
The figure is presented in a $\log_{10}$ scale to emphasize any mass spreading
in the domain, no matter its amount.
Figure~\ref{fig:mass_spread} shows no global mass spreading.
We observe only local mass spreading, which is expected, in form of smearing near
the cylinder, cone and bump.
In contrast, the OBFEM-DD scheme suffers from global mass spreading at the level of
$10^{-6}$ to $10^{-5}$, as shown in~\cite[Figure 21(d)]{peterson2024optimization},
where it is labeled OBT.

\begin{figure}
    \centering
    \includegraphics[width=.5\textwidth]{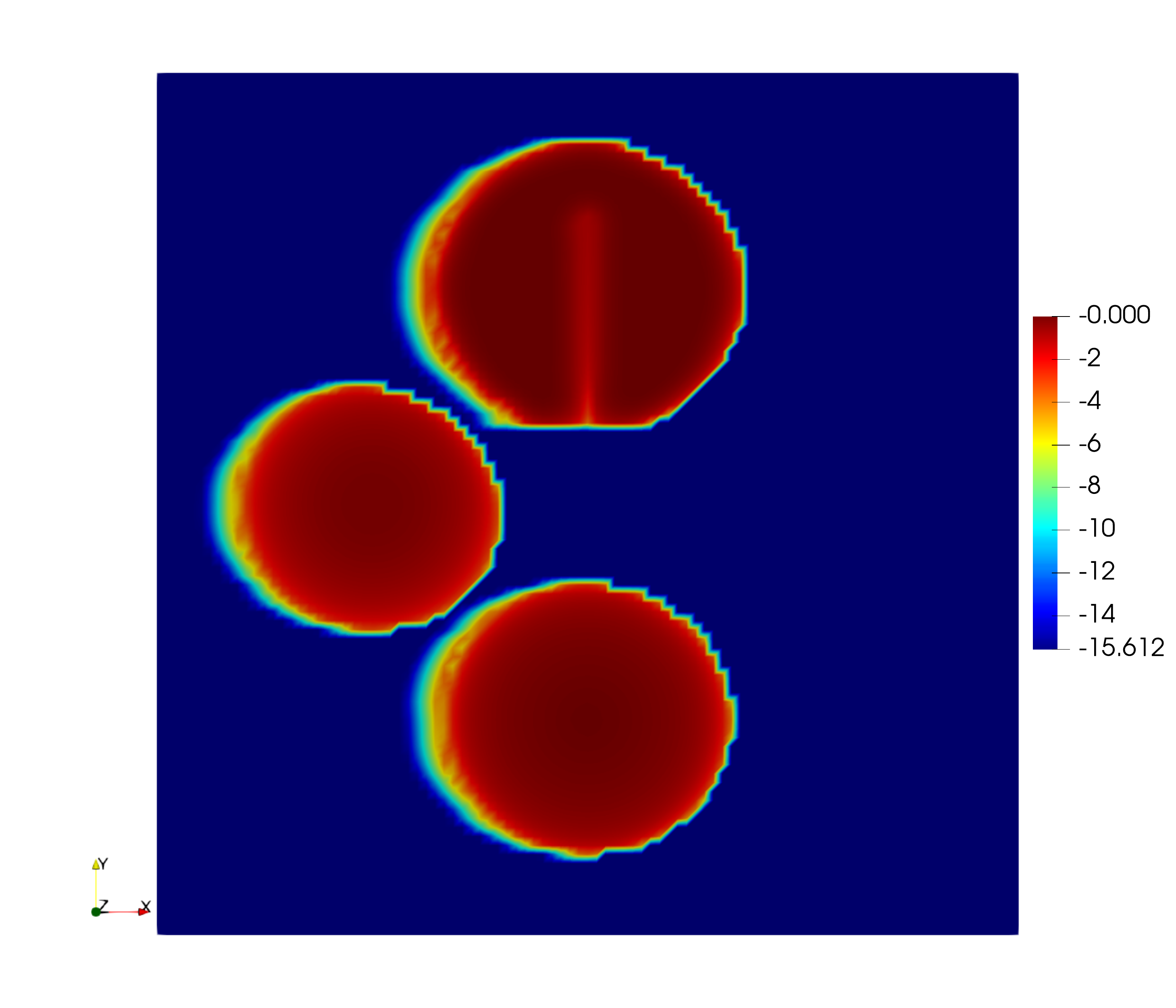}
    \caption{Log plot of the optimal solution $u_h^{\rm opt}$ at final time.}
    \label{fig:mass_spread}
\end{figure}

\subsection{Cahn-Hilliard equation}
To test our optimal control strategy in the context of an important nonlinear conservation law, we consider the Cahn-Hilliard equation, a fourth-order PDE that models phase separation processes in binary mixtures of incompressible fluids. The mixed form of the Cahn-Hilliard equation reads \cite{frank2020bound}
\begin{alignat*}{4}
    \pd{u}{t} - \nabla\cdot\left(\magenta{\eta}(u) \nabla\mu\right) &= 0 & & \text{ in } & & \hphantom{\partial}\Omega & & \times\left(0,T\right],\\
    \mu - \psi'(u) + \kappa \Delta u &= 0 & & \text{ in } & & \hphantom{\partial}\Omega & & \times\left(0,T\right],\\
    \nabla u \cdot \mathbf{n} &= 0 & & \text{ on } & & \partial\Omega & & \times\left(0,T\right],\\
    -\magenta{\eta}(u) \nabla \mu \cdot \mathbf{n} &= 0 & & \text{ on } & & \partial\Omega & & \times\left(0,T\right].
\end{alignat*}
The conserved quantity $u$ represents the difference between the mass fractions of the two components. The range of physically admissible values for $u$ is $\left[-1,1\right]$. The chemical potential $\mu$ is defined as the derivative of the free energy
$$ E(u) := \int_\Omega \left(\frac{\kappa}{2} \left | \nabla u \right |^2 + \psi(u)\right) \dx,$$
where $\kappa$ \magenta{determines} the width of the diffuse interface between the two fluids. In numerical experiments, $\kappa$ is bounded below by the mesh size $h$. The mobility of the separation process is controlled by $\magenta{\eta}(u)= \left(\max\left(0,1-u^2\right)\right)^k$ for $k\in\left\lbrace 0,1,2\right\rbrace$ in the cases of constant, degenerate and biquadratic-degenerate mobility. In our numerical studies, we use constant mobility and the polynomial potential
$$\psi\left(u\right)=\frac14\left(u^2-1\right)^2.$$

We discretize a weak form of the Cahn-Hilliard system in space using the continuous 
Galer\-kin method. At the semi-discrete level, the evolution of the
finite element approximations $u_h\approx u$
and $\mu_h\approx\mu$ is governed by
\begin{align*}
    \int_\Omega \varphi_i \pd{u_h}{t} \dx + \int_{\Omega} \magenta{\eta}(u_h) \nabla\varphi_i \cdot \nabla \mu_h \dx  &= 0,\qquad i=1,\ldots,N,\\
    \int_\Omega \varphi_i \mu_h \dx - \int_\Omega \varphi_i \psi'(u_h) \dx - \int_{\Omega}\kappa \nabla \varphi_i \cdot \nabla u_h \dx &= 0,\qquad i=1,\ldots,N.
\end{align*}
Following Shen and Yang \cite{shen2010numerical}, we discretize in time using a semi-implicit time stepping scheme. The resulting nonlinear system is given by
\begin{align} \label{CH:consistent}
    \begin{bmatrix} M_C & \Delta tA(\magenta{\eta}(u_h^n)) \\ -A(\kappa) & M_C \end{bmatrix} \begin{bmatrix} u^{n+1} \\ \mu^{n+1} \end{bmatrix} = \begin{bmatrix} M_C u^n \\ b(u_h^n)\end{bmatrix}.
\end{align}
The entries of $A(\magenta{\eta})=(a_{ij}(\magenta{\eta}))_{i,j=1}^N$ and $b(u)=(b_i(u))_{i=1}^N$ are defined as follows:
$$
a_{ij}(\magenta{\eta})=\int_{\Omega}\magenta{\eta}\nabla\varphi_i \cdot\nabla\varphi_j\dx,
\qquad b_i(u)=\int_{\Omega}\varphi_i\psi'(u)\dx.
$$

An advantage of this fully discrete scheme compared to a fully explicit one is the less severe time step restriction. In the case of constant mobility, the system matrix is constant and needs to be assembled just once. The vector $b(u_h^n)$ needs to be assembled once per time step. A fully implicit treatment would require solving a nonlinear system and incur a much higher computational cost.

We notice that the unconstrained
discrete problem is equivalent to
\begin{align} \label{CH:obfem}
  \begin{bmatrix} M_L & \Delta tA(\magenta{\eta}(u_h^n)) \\ -A(\kappa) & M_C \end{bmatrix} \begin{bmatrix} u^{n+1} \\ \mu^{n+1} \end{bmatrix} = \begin{bmatrix}
    M_L u^n+ \Delta t (M_L-M_C)\dot u^{n+1} \\ b(u_h^n)\end{bmatrix}
\end{align}
with $\dot u^{n+1}=\frac{u^{n+1} - u^n}{\Delta t}$. Treating this choice as a target,
we can now enforce preservation of global bounds $u^{\min}=-1$ and
$u^{\max}=1$ by optimizing $\dot u^{n+1}$.

\subsubsection{Merging droplets in 2D} \label{test:ch2}
We begin with the two-dimensional merging droplets problem considered in \cite{liu2019numerical}. The initial condition for the solution of the Cahn-Hilliard equation is given by
$$ u_0(x,y) = \begin{cases} 1 &\text{if} \left(\left(\frac{1}{8} \leq x \leq \frac{1}{2} \right)\wedge \left(\frac{1}{8} \leq y \leq \frac{1}{2}\right)\right) \\ &\vee\left(\left(\frac{1}{2} \leq x \leq \frac{7}{8} \right)\wedge \left(\frac{1}{2} \leq y \leq \frac{7}{8}\right)\right), \\ -1 &\text{otherwise}.\end{cases} $$
This initial configuration is shown in Figure~\ref{fig:md_init}. Two squares are filled with fluid $a$, which is surrounded by fluid $b$. We use constant mobility, $k=0$, and the interface width $\kappa = 2^{-10}$. The computational domain is discretized using a uniform $64\times 64$ grid. The numerical solution is initialized by the lumped-mass $L^2$ projection of $u_0$. We use the time step $\Delta t = 10^{-4}$ and stop at the final time $T=0.1$, by which the two squares merge and a diffuse interface forms between the fluids. This result is consistent with the findings of Liu et al. \cite{liu2019numerical}.

\begin{figure}
	\begin{subfigure}{0.495\textwidth}
	    \includegraphics[width=\textwidth]{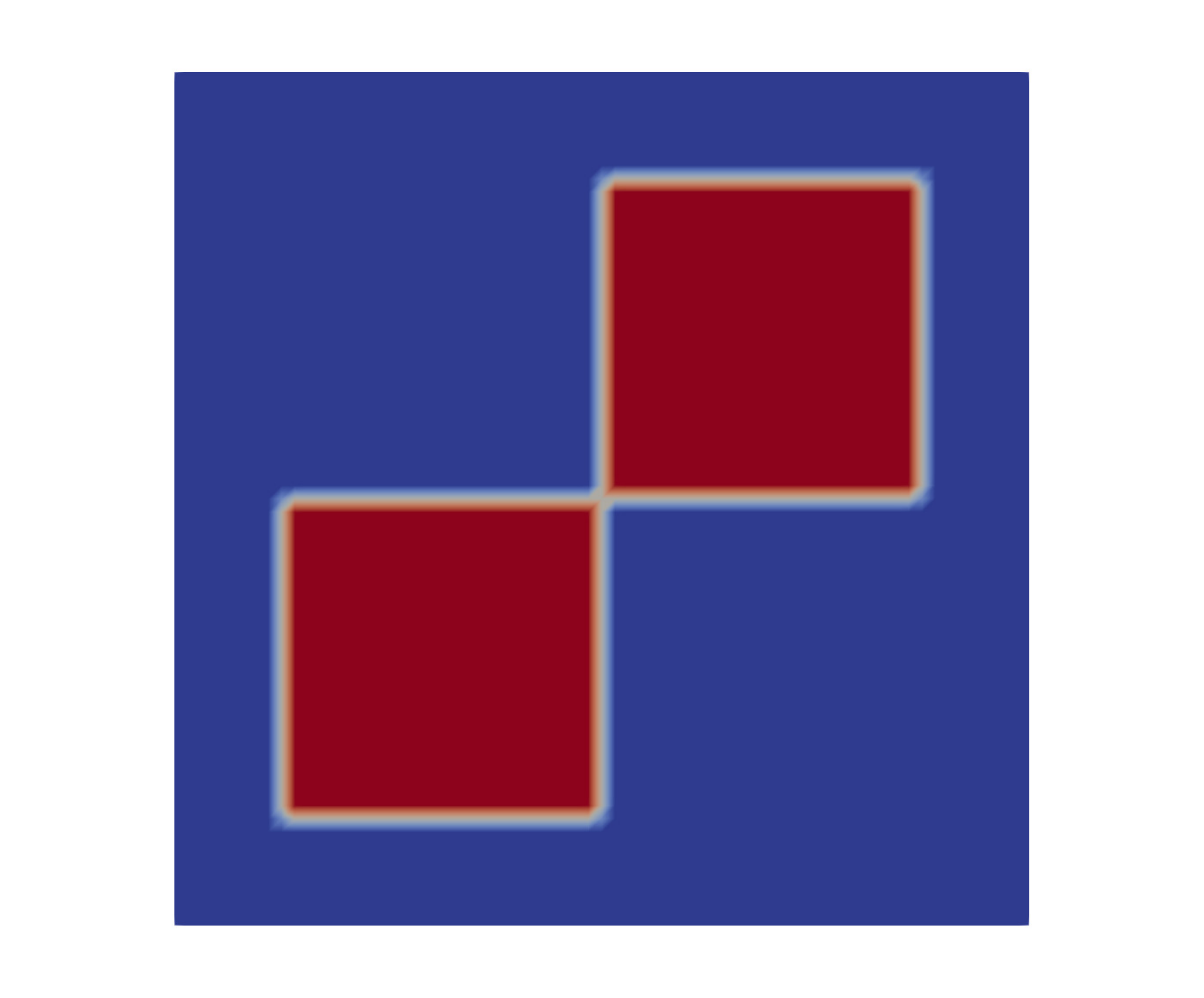}
    	\caption{Projected initial condition;\\ $u_h^0 \in \left[-1,1 \right]$.}
    	\label{fig:md_init}
	\end{subfigure}
	\hfill
    \begin{subfigure}{0.495\textwidth}
        \includegraphics[width=\textwidth]{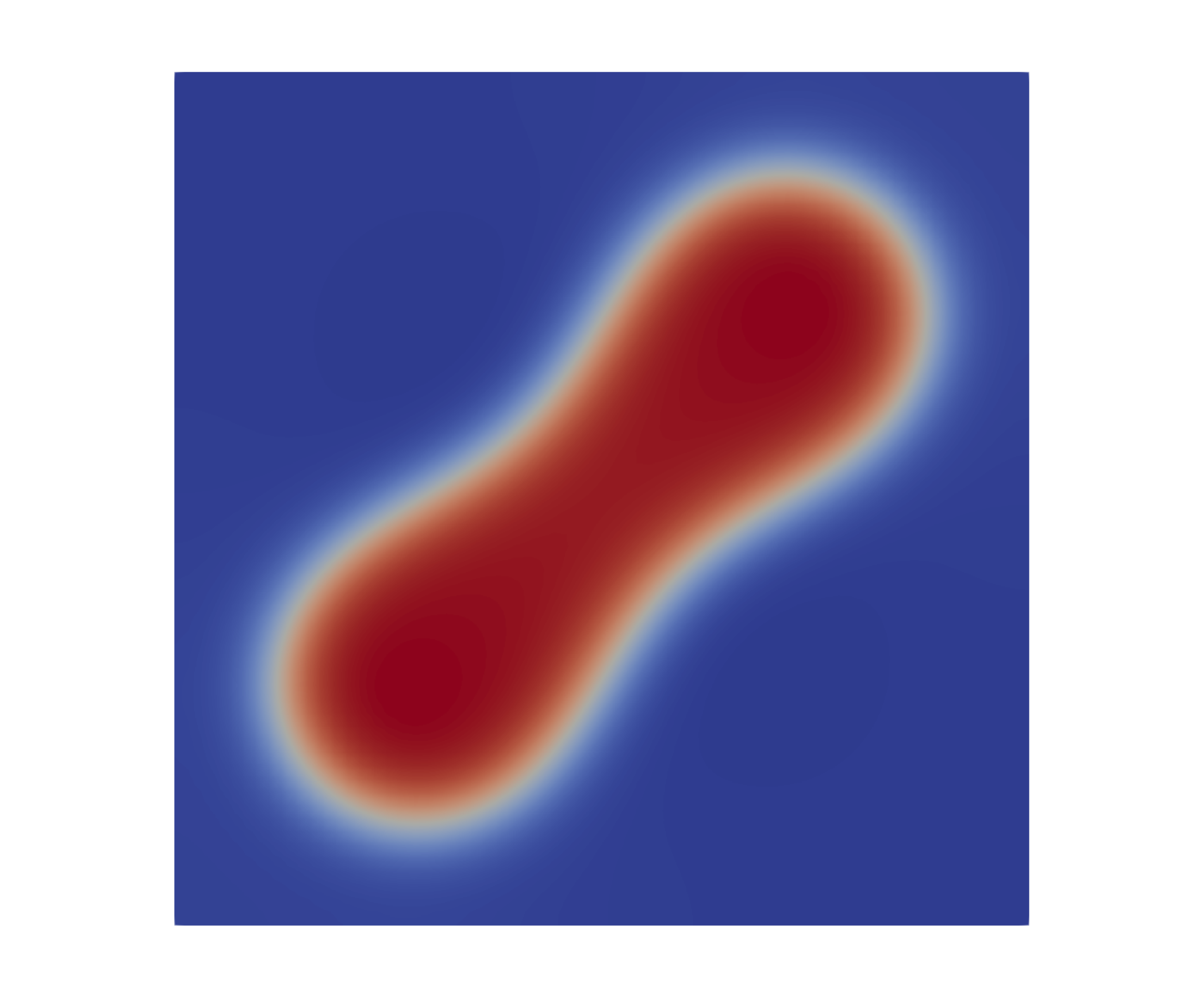}
        \caption{Target solution;\\ $u_h^H \in \left[-0.996,1.024 \right]$.}
        \label{fig:md_ref}
    \end{subfigure}
    \vskip\baselineskip
    \begin{subfigure}{0.495\textwidth}
        \includegraphics[width=\textwidth]{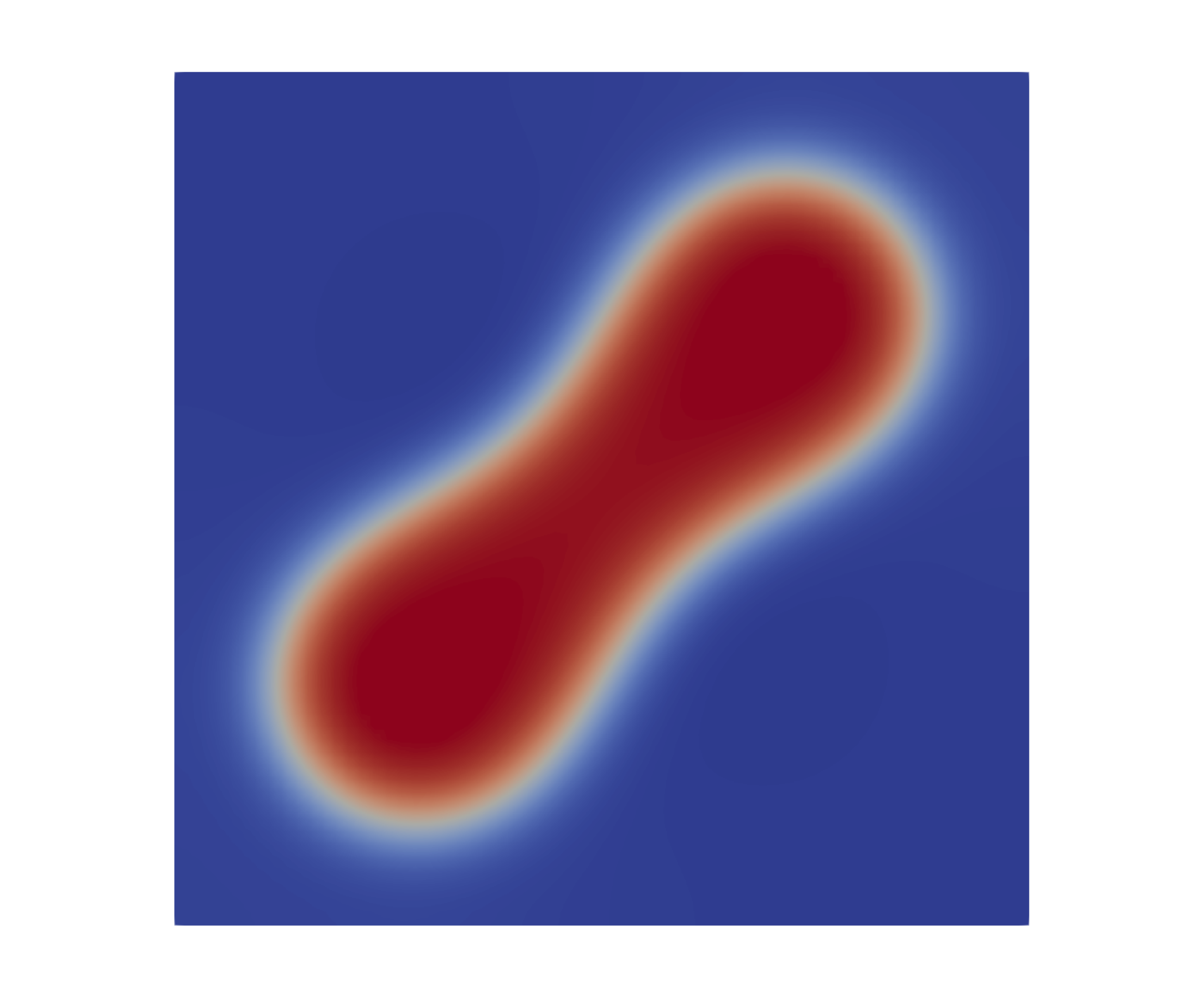}
        \caption{Optimal control solution;\\ $u_h^{\text{opt}} \in \left[-0.996,1\right]$.}
        \label{fig:md_opt}
    \end{subfigure}
    \hfill
    \begin{subfigure}{0.495\textwidth}
        \includegraphics[width=\textwidth]{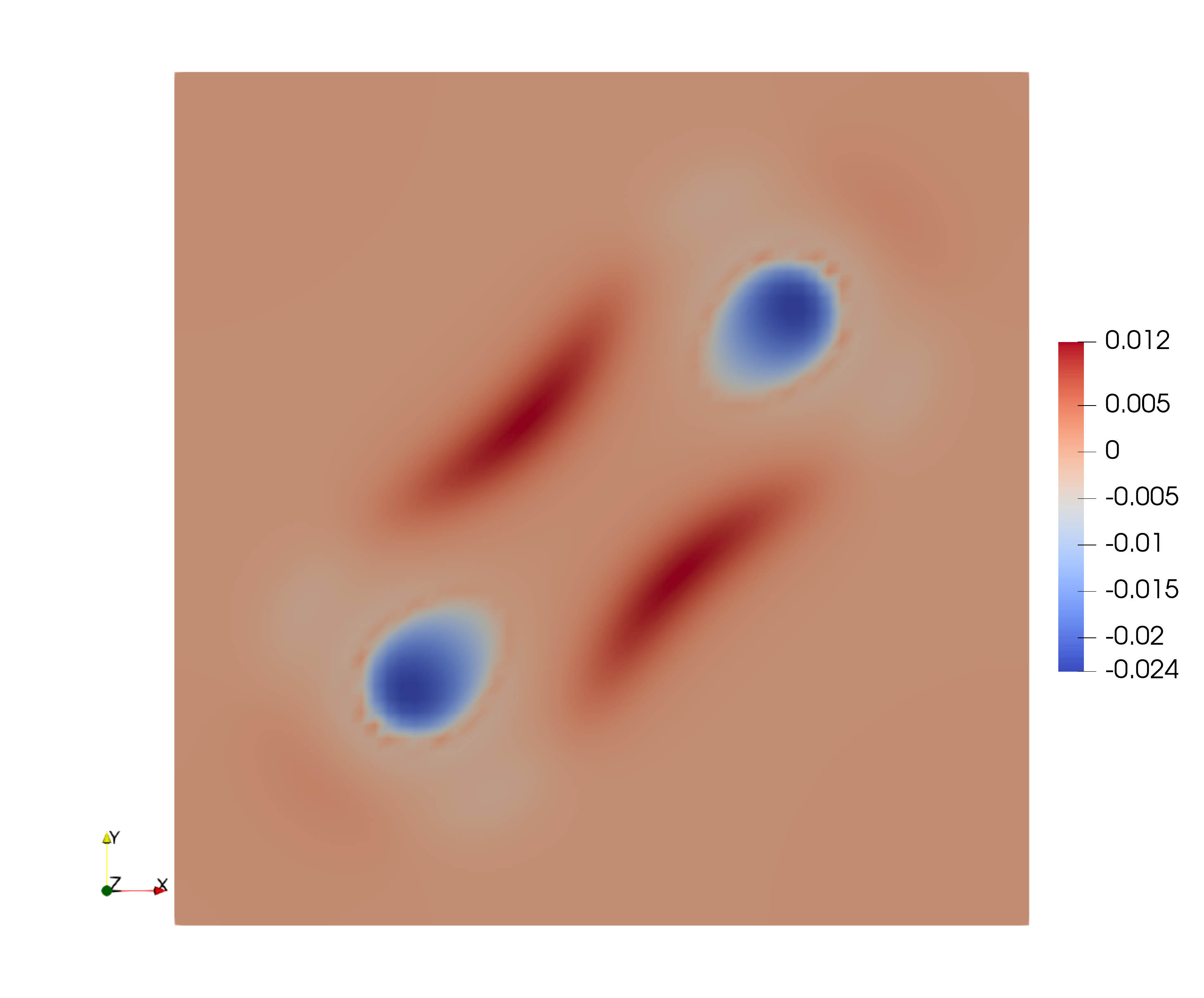}
        \caption{Difference between the solutions;\\\centering $u_h^{\text{opt}} - u_h^H \in \left[-0.024, 0.012 \right]$.}
        \label{fig:md_dif}
    \end{subfigure}
    \caption{Approximate solutions of the 2D merging droplets problem.}
\end{figure}

The numerical solution $u_h^H$ that we show in Figure \ref{fig:md_ref} was
obtained with the unconstrained target scheme \eqref{CH:consistent}. It violates the upper bound $u^{\max}=1$ but the global minimum $-0.996$ stays in the admissible range. Figure \ref{fig:md_opt} shows the bound-preserving optimal control solution $u_h^{\rm opt}$. Interestingly, it has the same acceptable global minimum as $u_h^H$. The difference between the two approximations is plotted in Figure~\ref{fig:md_dif}. It is particularly large in the blue subdomains, in which $u_h^H$ has overshoots, while $u_h^{\rm opt}$ is bounded above by $1$. 

\red{In order to examine the performance of the proposed algorithm, we conduct a scaling analysis motivated by the convergence tests in~\cite{liu2019numerical}.
For $i \in \left\{3,\dots, 7\right\}$, the step size pairs $(h_i,\Delta t_i)$ are tested, with $h_i = 2^{-i}$ and $\Delta t_i = 0.1 \cdot 2^{-2i}$.
The results are presented in Table~\ref{tab:scal_conv_ch}.
Therein the numbers of function evaluations, gradient evaluations and projections increase slightly with mesh refinement.
The final criticality measures, $\|\nabla\|$, increase due to our use of a relative stopping tolerance and the increasing initial criticality measures as the mesh is refined.
In the scaling benchmark from Section~\ref{test:rm} the iteration numbers decline for finer meshes due to a better sampling of the discontinuities.
The bound violations of the consistent projection become less severe, yielding faster solutions.
In contrast, in this example the violations \emph{remain} on finer meshes, resulting in iteration numbers that are inversely proportional to the mesh size~$h_i$.
We note that the numbers of optimization variables are inversely proportional to~$h_i^2$.
Hence, the algorithm performs well under mesh refinement.}

\begin{table}\centering
  \begin{tabular}{cccccccc}
  \hline
  index $i$ & mesh size & \#f & \#g & \#p & mass err & $\|\nabla\|$ \\ 
  \hline
  3 & 1 / 8   & \hphantom{1}4.75 & \hphantom{1}4.75 & \hphantom{1}33.19 & 2.36e-16 & 1.75e-06  \\
  4 & 1 / 16  & \hphantom{1}4.74 & \hphantom{1}4.74 & \hphantom{1}31.84 & 3.26e-17 & 6.58e-06  \\
  5 & 1 / 32  & \hphantom{1}5.59 & \hphantom{1}5.59 & \hphantom{1}37.46 & 4.83e-17 & 6.40e-05  \\
  6 & 1 / 64  & \hphantom{1}9.78 & \hphantom{1}9.78 & \hphantom{1}73.40 & 7.67e-17 & 4.07e-04  \\
  7 & 1 / 128 &            18.35 &            18.35 &            140.49 & 8.42e-17 & 2.91e-03  \\
  \hline
  \end{tabular}
  \caption{Impact of mesh refinement on the scalability of optimal control for example \ref{test:ch2}.}
  \label{tab:scal_conv_ch}
\end{table}

\subsubsection{Merging droplets in 3D} \label{test:ch3}
In the next example, we solve a three-dimensional merging droplets problem. Following Liu et al.
\cite{liu2023}, we choose the initial condition 
\begin{align*}
u_0(\mathbf x) = \max \left\{-1, \tanh\left(50\frac{0.25-\left\| \mathbf x - \mathbf a_0\right\|}{\sqrt2}\right),\tanh\left(50\frac{0.25-\left\| \mathbf x - \mathbf a_1\right\|}{\sqrt2}\right)\right.&,  \\
	   \left. \tanh\left(50\frac{0.16-\left\| \mathbf x - \mathbf a_2\right\|}{\sqrt2}\right),\tanh\left(50\frac{0.16-\left\|\mathbf x -\mathbf a_3\right\|}{\sqrt2}\right) \right\}&,
\end{align*}
where $\mathbf a_0 = \left[0.35,0.35,0.35 \right]^\top$, $\mathbf a_1 = \left[0.65,0.65,0.65 \right]^\top$, $\mathbf a_2 = \left[0.75,0.25,0.25 \right]^\top$, and $\mathbf a_3 = \left[0.25,0.75,0.75 \right]^\top$ are the centers of four droplets.

We compute numerical solutions using constant mobility and $\kappa = \frac{1}{50^2}$. The final time is $T=1.5s$. To visualize the regions where $u_h > 0$, we use Paraview with a clip filter. Figure \ref{fig:md3_init} shows the lumped-mass $L^2$ projection of the initial condition. The global maximum of the target solution shown in Figure \ref{fig:md3_ref} exceeds the upper bound by $0.028$. The optimal control solution, which we present in Figure \ref{fig:md3_opt}, does preserve the global bounds. To better visualize the difference between the two solutions , we cut the 3D plot in Figure~\ref{fig:md3_dif} using a hyperplane orthogonal to $\mathbf{n}=\left(-1,-1,1\right)^\top$. The global minimum $-0.028$ of $u_h^{\text{opt}} - u_h^H$ has the same magnitude as the overshoot in $u_h^H$. The dark purple disk in the middle visualizes the violations of the upper bound. The surrounding yellow circle is the transition zone that receives the mass from the overshoot region. 
\begin{figure}
    \begin{subfigure}{0.495\textwidth}
        \includegraphics[width=\textwidth]{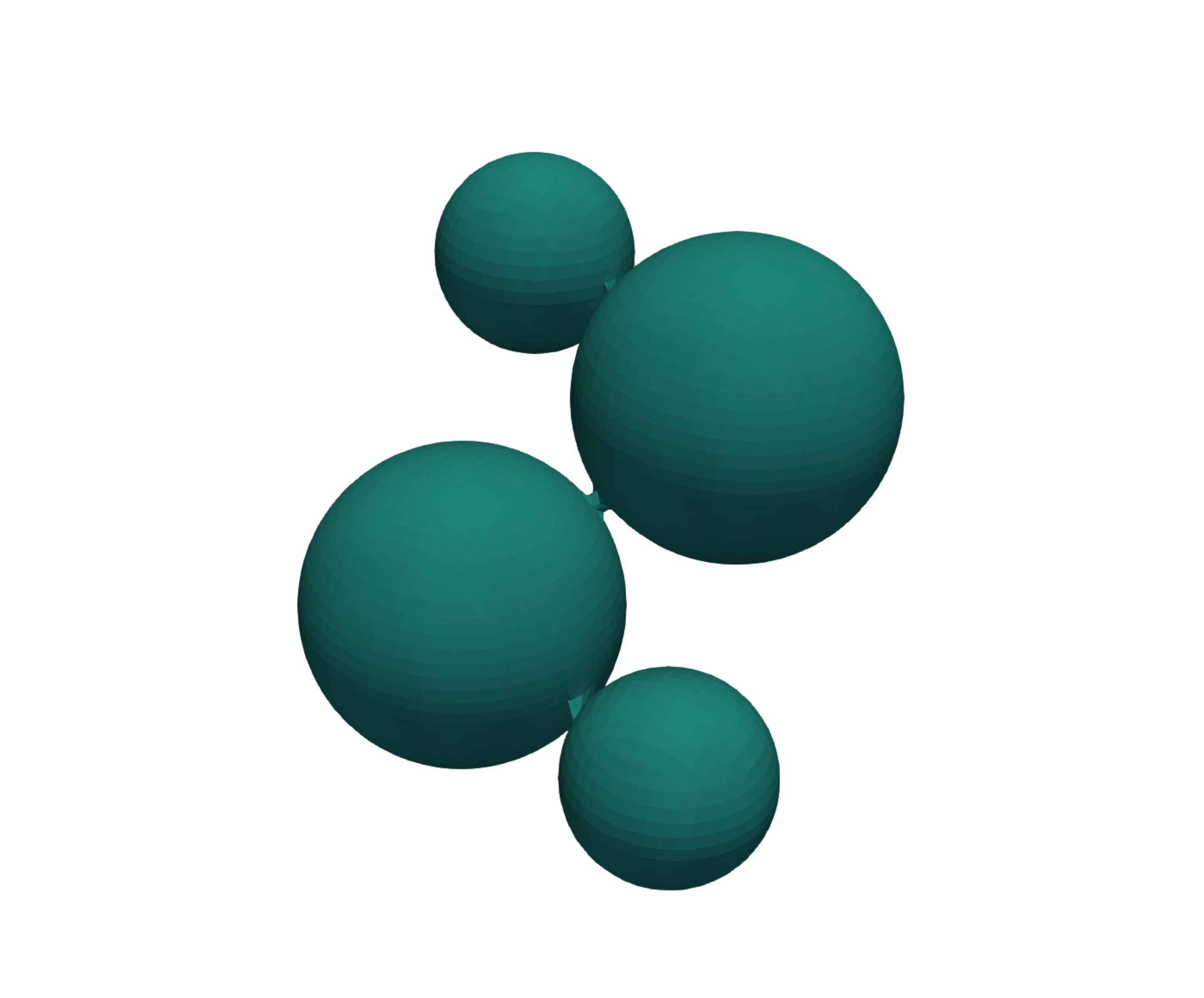}
        \caption{Projected initial condition;\\ $u_h^0 \in \left[-1,1 \right]$.}
        \label{fig:md3_init}
    \end{subfigure}
    \hfill
    \begin{subfigure}{0.495\textwidth}
        \includegraphics[width=\textwidth]{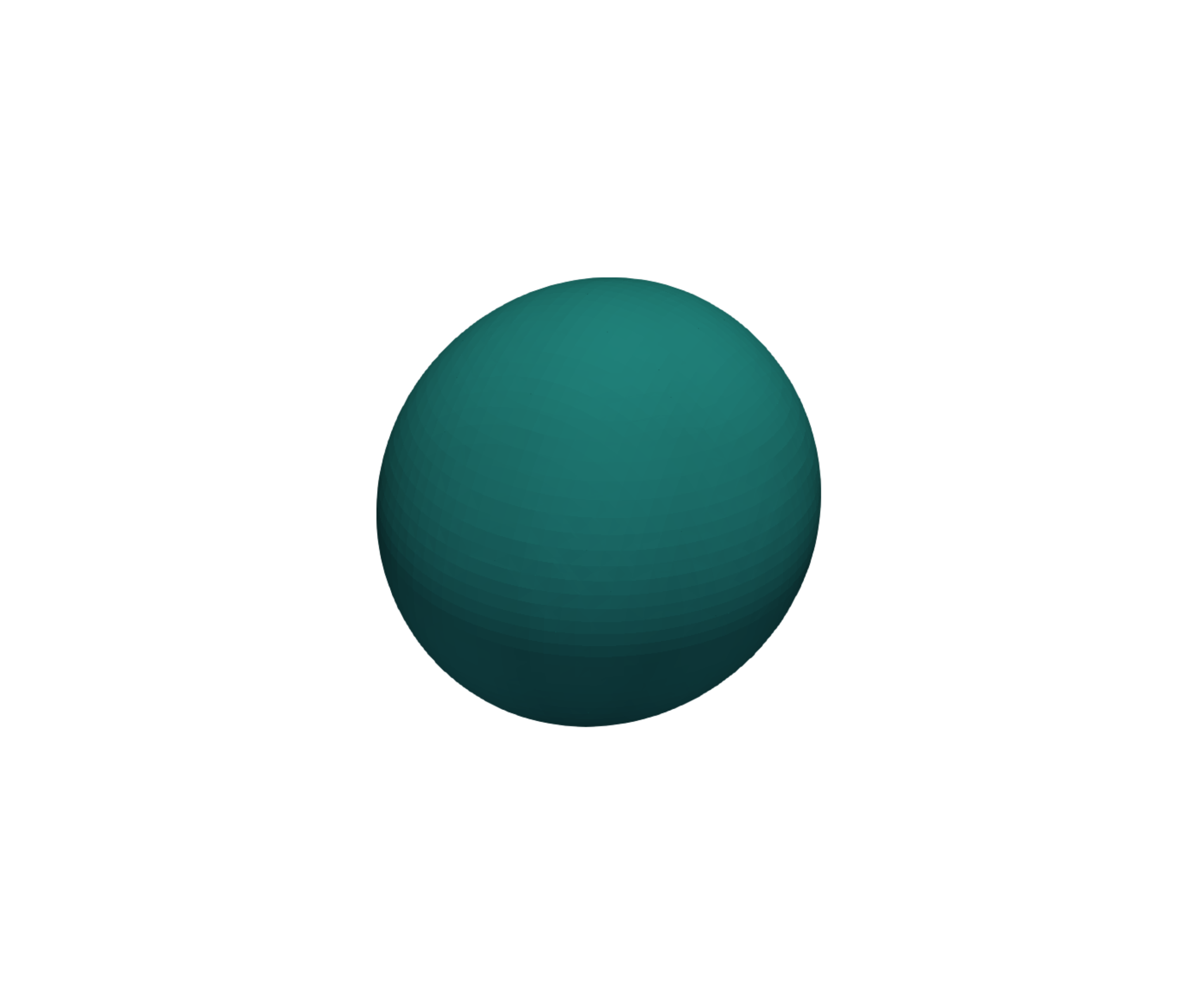}
        \caption{Target solution;\\ $u_h^H \in \left[-0.971,1.028\right]$.}
        \label{fig:md3_ref}
    \end{subfigure}
    \vskip\baselineskip
    \begin{subfigure}{0.495\textwidth}
        \includegraphics[width=\textwidth]{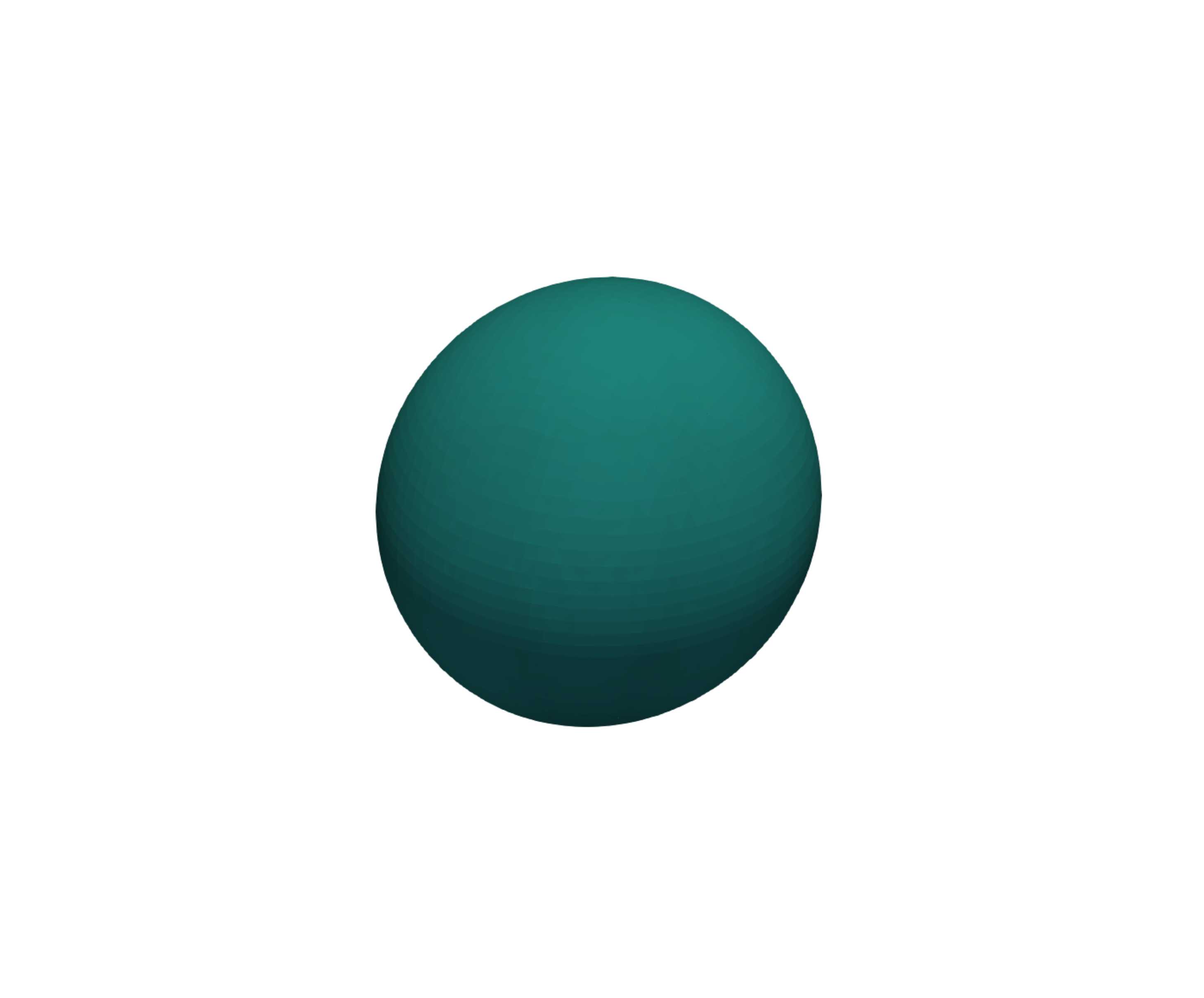}
        \caption{Optimal control solution;\\ $u^{\text{opt}} \in \left[-0.971,1\right]$.}
        \label{fig:md3_opt}
    \end{subfigure}
    \hfill
    \begin{subfigure}{0.495\textwidth}
        \includegraphics[width=\textwidth]{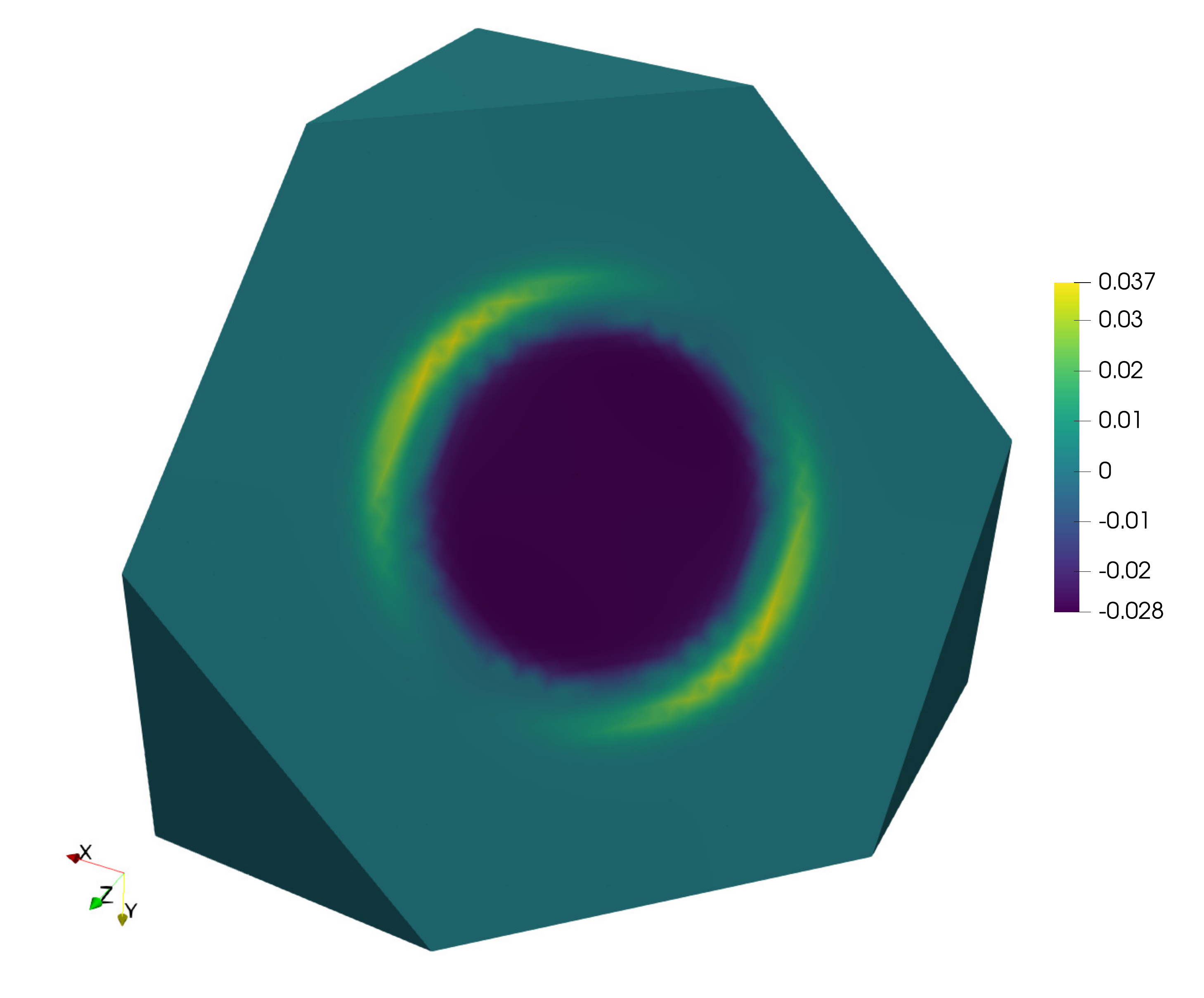}
        \caption{Difference between the solutions;\\\centering $u_h^{\text{opt}} - u_h^H \in \left[-0.028, 0.037 \right]$.}
        \label{fig:md3_dif}
    \end{subfigure}
    \caption{Approximate solutions of the 3D merging droplets problem.}
\end{figure}

\red{Based on the scaling test from Section \ref{test:ch2}, we conduct a similar test with the step size pairs $(h_i, \Delta t_i)$ for $i \in \left\{3,\dots,6 \right\}$ and $\kappa=2^{-10}$.
The results are presented in Table \ref{tab:scal_conv_ch2}, where a slight increase in the function evaluation, gradient evaluation and projection numbers is observed.
We note that the rate of increase is sublinear with respect to $1/h_i$, even though the numbers of optimization variables increase cubically, i.e., as~$1/h_i^3$. These results demonstrate nearly scalable performance of our approach.}

\begin{table}\centering
  \begin{tabular}{cccccccc}
  \hline
  index $i$ & mesh size & \#f & \#g & \#p & mass err & $\|\nabla\|$ \\ 
  \hline
  3 & 1 / 8   & 4.03 & 4.03 & 27.16 & 3.43e-17 & 2.71e-06  \\
  4 & 1 / 16  & 4.38 & 4.38 & 28.82 & 2.37e-17 & 2.04e-05  \\
  5 & 1 / 32  & 5.80 & 5.80 & 39.36 & 1.77e-17 & 7.11e-05  \\
  6 & 1 / 64  & 9.71 & 9.71 & 66.79 & 2.16e-16 & 7.46e-04  \\
  \hline
  \end{tabular}
  \caption{Impact of mesh refinement on the scalability of optimal control for example \ref{test:ch3}.}
  \label{tab:scal_conv_ch2}
\end{table}

\subsubsection{Spinodal decomposition}
In the last example, we use the Cahn-Hilliard equation with $\kappa=2^{-10}$
to simulate
spinodal decomposition of two fluids. Our 3D version of the 2D test
problem considered in \cite{frank2020bound} defines the initial condition
\begin{align*}
    u_0\left(x,y,z\right) \in \left\{ -0.99, 0, 0.99 \right\}
\end{align*}
by  choosing from equally distributed random numbers. The lumped-mass $L^2$ projection of $u_0$ is depicted in Figure~\ref{fig:sd_init}. In this 3D experiment, we use a uniform mesh consisting of  $64\times64\times64$ cubes and the time step $\Delta t = 10^{-4}$. The simulation is terminated at the final time $T=0.1$. The target solution shown in Figure~\ref{fig:sd_ref} exhibits undershoots and overshoots, whereas the optimal control solution shown in Figure~\ref{fig:sd_opt} preserves the global bounds. A hyperplane cut of the difference between the two solutions is presented in Figure~\ref{fig:sd_dif}. The plot confirms that this difference is small in regions where $u_h^H$ is bound preserving.
\begin{figure}
    \begin{subfigure}{0.495\textwidth}
        \includegraphics[width=\textwidth]{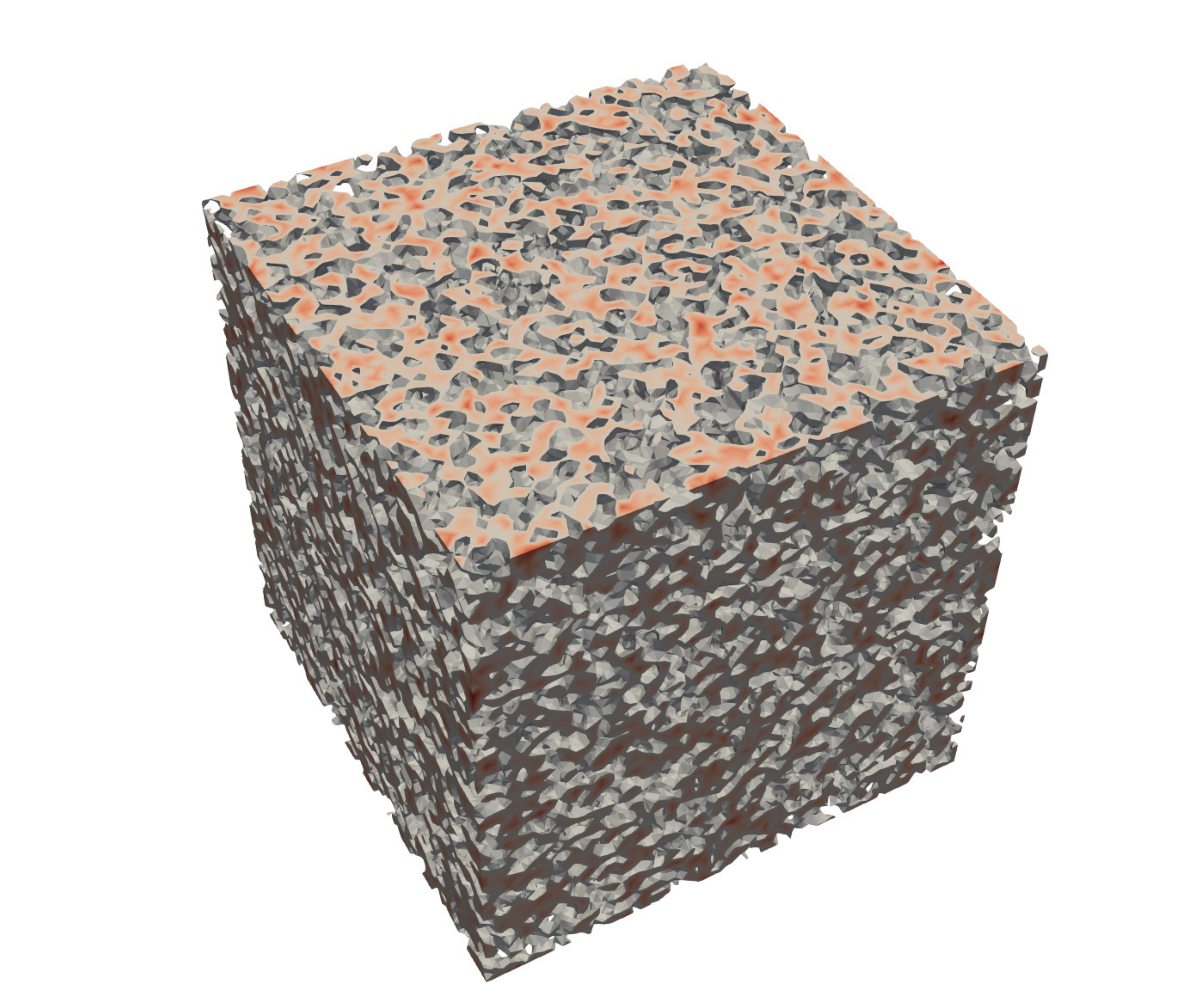}
        \caption{Projected initial condition;\\ $u_h^0 \in \left[-0.838,0.841 \right]$.}
        \label{fig:sd_init}
    \end{subfigure}
    \hfill
    \begin{subfigure}{0.495\textwidth}
        \includegraphics[width=\textwidth]{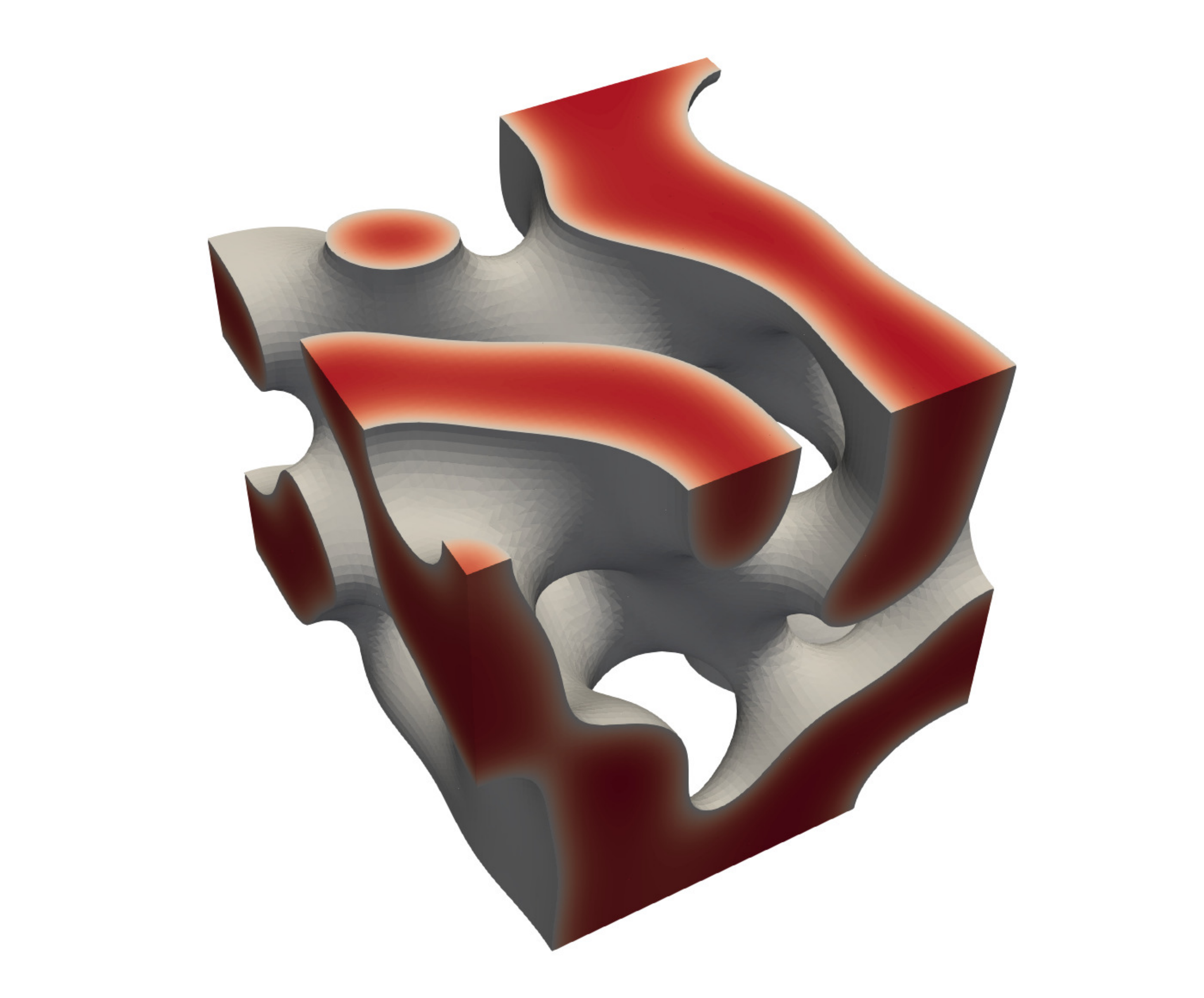}
        \caption{Target solution;\\ $u_h^H \in \left[-1.052,1.028 \right]$.}
        \label{fig:sd_ref}
    \end{subfigure}
    \vskip\baselineskip
    \begin{subfigure}{0.495\textwidth}
        \includegraphics[width=\textwidth]{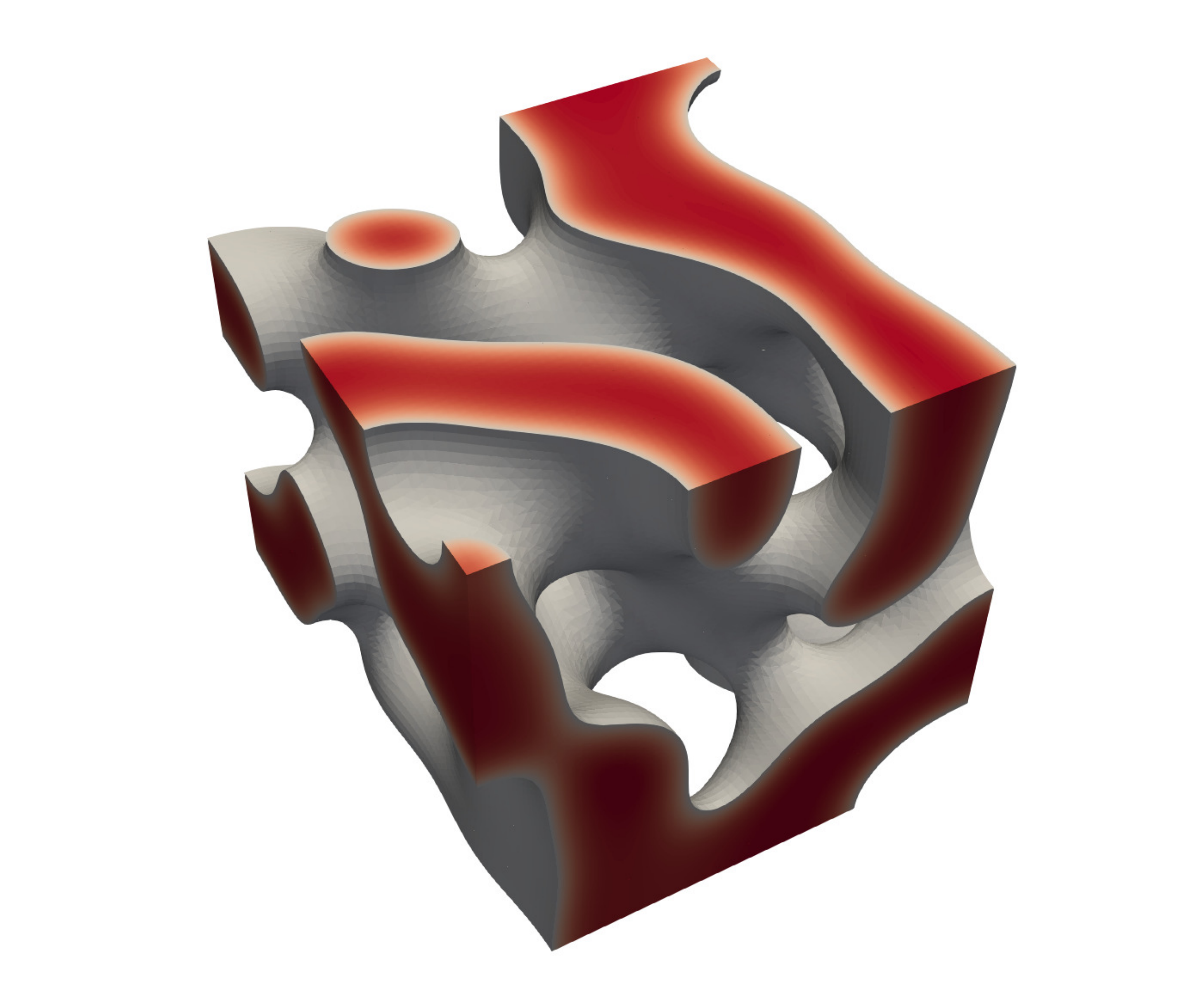}
        \caption{Optimal control solution;\\ $u_h^{\text{opt}} \in \left[-1,1\right]$.}
        \label{fig:sd_opt}
    \end{subfigure}
    \hfill
    \begin{subfigure}{0.495\textwidth}
        \includegraphics[width=\textwidth]{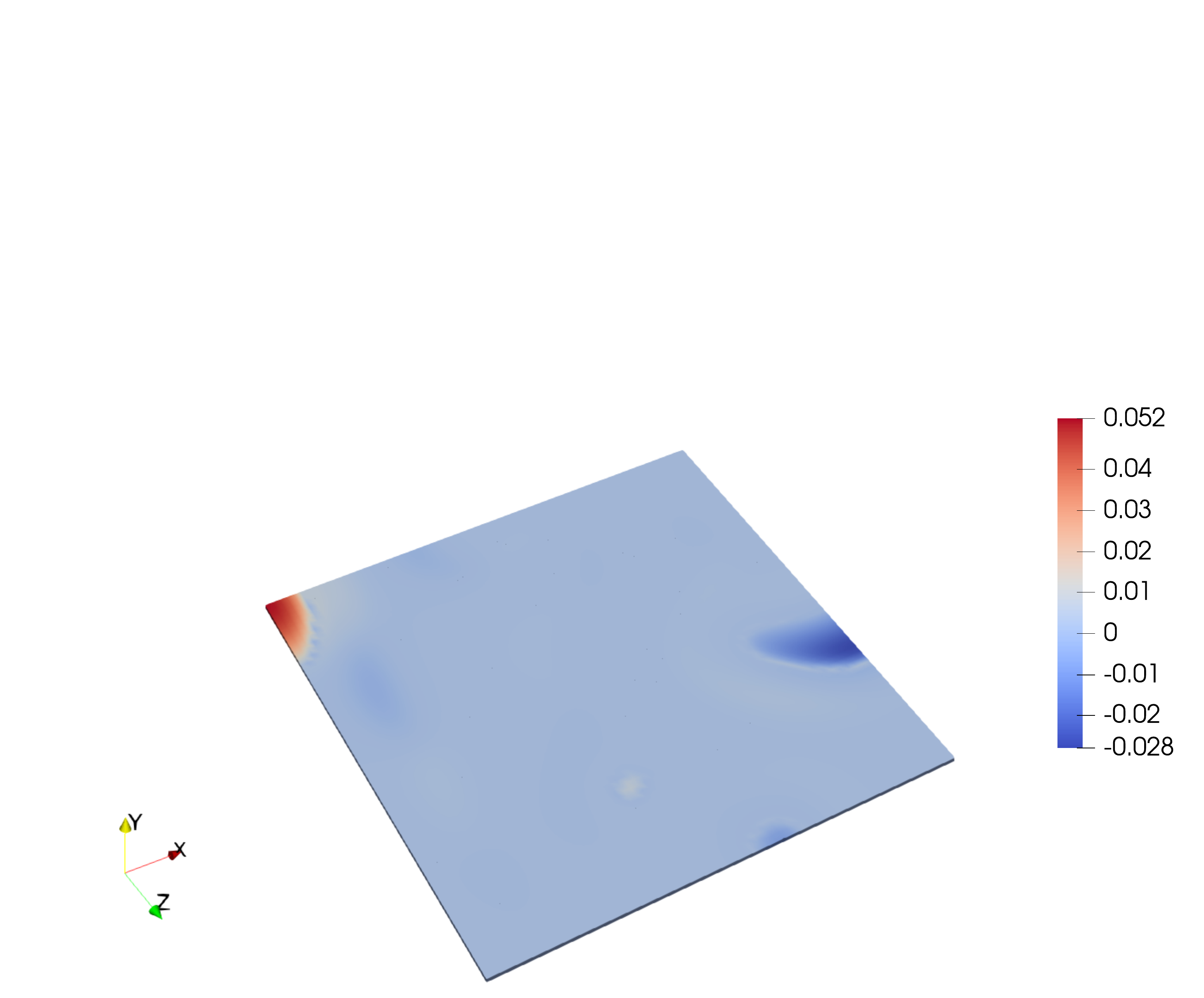}
        \caption{Difference between the solutions;\\\centering $u_h^{\text{opt}} - u_h^H \in \left[-0.028, 0.052 \right]$.}
        \label{fig:sd_dif}
    \end{subfigure}
    \caption{Approximate solutions of the spinodal decomposition problem.}
\end{figure}

\section{Conclusions}
\label{sec:conclusions}
\PBB{In this paper, we presented a new potential-target optimization-based property preserving scheme for scalar conservation laws. The proposed approach improves upon existing OB formulations in two important ways. First, it constrains finite element approximations to satisfy discrete maximum principles while preserving good local conservation properties, thereby avoiding issues such as the mass spreading in OBFEM-DD. Second, by using a special semi-norm and a diagonal scaling operator in the definition of the optimization problem, our formulation achieves  fast convergence of iterative solvers.}
The numerical scheme combines inexact multigrid solvers with very loose stopping tolerances, used to evaluate the objective function and its derivatives, with near-exact projections onto convex sets, used to satisfy the constraints resulting from the discrete maximum principles.
Our numerical studies of representative test problems indicate that the inexact evaluation of objective functions, gradients, and Hessians may also result in remarkable efficiency gains. The combination of superb accuracy, reliability, and computational efficiency makes our optimal control approach an attractive alternative to \PBB{traditional} closed-form limiters.

\section*{Acknowledgments}
The authors are honored by the invitation to submit this contribution to the Special Issue of MCRF dedicated to Arnd Rösch's 60th birthday
in 2025.

The work of Falko Ruppenthal and Dmitri Kuzmin was supported by the
German Research Foundation (DFG) under grant KU 1530/29-1.
The work of Pavel Bochev and Denis Ridzal contributing to this material
was supported by the U.S.\ Department of Energy, Office of Science, Office
of Advanced Scientific Computing Research under Field Work Proposal 23-020467.

Sandia National Laboratories is a multi-mission laboratory managed and operated
by National Technology \& Engineering Solutions of Sandia, LLC (NTESS), a wholly
owned subsidiary of Honeywell International Inc., for the U.S. Department of
Energy’s National Nuclear Security Administration (DOE/NNSA) under contract DE-NA0003525.
This written work is authored by an employee of NTESS. The employee, not NTESS, owns
the right, title and interest in and to the written work and is responsible for
its contents. Any subjective views or opinions that might be expressed in the
written work do not necessarily represent the views of the U.S.\ Government.
The publisher acknowledges that the U.S.\ Government retains a non-exclusive,
paid-up, irrevocable, world-wide license to publish or reproduce the published
form of this written work or allow others to do so, for U.S.\ Government purposes.
The DOE will provide public access to results of federally sponsored research in
accordance with the DOE Public Access Plan.

\bibliographystyle{amsplain}
\bibliography{literature.bib}

\medskip
Received xxxx 20xx; revised xxxx 20xx; early access xxxx 20xx.
\medskip

\end{document}